\UseRawInputEncoding 

\documentclass[10pt]{article}
\usepackage[letterpaper, left=1in, top=1in, right=1in, bottom=1in, verbose, ignoremp]{geometry}

\usepackage{url}\RequirePackage[colorlinks,citecolor=blue, linkcolor=blue,urlcolor = blue]{hyperref}
\usepackage{latexsym,amssymb,amsmath,amsfonts,graphicx,color,fancyvrb,amsthm,enumerate,subcaption,mathrsfs}
\usepackage[longnamesfirst,authoryear,round]{natbib}
\usepackage[dvipsnames]{xcolor}
\usepackage{xy}\xyoption{all} \xyoption{poly} \xyoption{knot}
\usepackage{float}
\usepackage{anyfontsize}%
\usepackage[title]{appendix}%
\usepackage{textcomp}%
\usepackage{manyfoot}%
\usepackage{rotating}
\usepackage{booktabs}%
\usepackage{algorithm}%
\usepackage{diagbox}
\usepackage{algorithmicx}%
\usepackage{listings}%
\usepackage{bm}
\usepackage{bbm}
\usepackage{multicol,multirow}
\usepackage{array}
\usepackage{relsize}
\usepackage{chngcntr}
\usepackage{etoolbox}
\usepackage{caption}
\usepackage{tikz}
\usetikzlibrary{decorations.pathreplacing,calc}
\usetikzlibrary{shapes,backgrounds}
\usetikzlibrary{patterns}
\usetikzlibrary{cd}
\usepackage{amsopn}
\usepackage{calc}
\usepackage{algpseudocode}
\usepackage{mathtools}

\usepackage{tikz-cd} 
\usepackage{amscd} 
\usepackage{comment}
\usepackage[capitalize,nameinlink,compress]{cleveref}
\crefname{figure}{Figure}{Figures} 
\crefname{equation}{}{} 
\crefname{assumption}{Assumption}{Assumptions}
\crefname{subsection}{Subsection}{Subsections}
\crefname{open-prob}{Open Problem}{Open Problems}

\newcounter{cdrow}

\newtheorem{theorem}{Theorem}[]
\newtheorem*{theorem*}{Theorem}
\newtheorem{corollary}[theorem]{Corollary}
\newtheorem{lemma}[theorem]{Lemma}
\newtheorem{proposition}[theorem]{Proposition}

\newtheorem*{claim*}{Claim}
\newtheorem{open-prob}[theorem]{Open Problem}

\theoremstyle{definition}
\newtheorem{definition}[theorem]{Definition}
\newtheorem*{definition*}{Definition}

\theoremstyle{remark}
\newtheorem{remark}[theorem]{Remark}

\newtheorem{example}[theorem]{Example}
\newtheorem*{example*}{Example}

\DeclareMathOperator{\tr}{tr}

\def\tr{{\rm tr}}

\makeatletter
\newcommand*{\op}{%
	\DOTSB
	\mathop{\vphantom{\bigoplus}\mathpalette\matt@op\relax}%
	\slimits@
}
\newcommand\matt@op[2]{%
	\vcenter{\m@th\hbox{\resizebox{\widthof{$#1\bigoplus$}}{!}{$\boxplus$}}}%
}
\makeatother

\newcommand{\R}{\mathbb{R}}
\newcommand{\N}{\mathbb{N}}

\renewcommand{\Im}{\mathrm{Im} \,}

\newcommand{\argmin}{\text{argmin}}
\newcommand{\TPT}[1]{ \R^{#1}/\R \mathbf{1}}
\newcommand{\dtr}{d_{\mathrm{tr}}}

\newcommand{\sgn}{\text{\sgn}}

\newcommand{\E}{\mathbb{E}}

\renewcommand{\vec}[1]{\mathbf{#1}}

\renewcommand{\ker}{\text{Ker}}

\DeclareMathOperator{\type}{type}

\makeatletter
\def\@biblabel#1{}
\makeatother

\makeatletter
\patchcmd{\NAT@citex}
{\@citea\NAT@hyper@{%
		\NAT@nmfmt{\NAT@nm}%
		\hyper@natlinkbreak{\NAT@aysep\NAT@spacechar}{\@citeb\@extra@b@citeb}%
		\NAT@date}}
{\@citea\NAT@nmfmt{\NAT@nm}%
	\NAT@aysep\NAT@spacechar\NAT@hyper@{\NAT@date}}{}{}

\patchcmd{\NAT@citex}
{\@citea\NAT@hyper@{%
		\NAT@nmfmt{\NAT@nm}%
		\hyper@natlinkbreak{\NAT@spacechar\NAT@@open\if*#1*\else#1\NAT@spacechar\fi}%
		{\@citeb\@extra@b@citeb}%
		\NAT@date}}
{\@citea\NAT@nmfmt{\NAT@nm}%
	\NAT@spacechar\NAT@@open\if*#1*\else#1\NAT@spacechar\fi\NAT@hyper@{\NAT@date}}
{}{}

\makeatother

\begin{document}
	
	\def\spacingset#1{\renewcommand{\baselinestretch}%
		{#1}\small\normalsize} \spacingset{1}

	\begin{flushleft}
		{\Large{\textbf{Probability Metrics for Tropical Spaces of Different Dimensions}}}
		\newline
		\\
		Roan Talbut$^{1,\dagger}$, Daniele Tramontano$^{2}$, Yueqi Cao$^{1}$, Mathias Drton$^{2}$, and Anthea Monod$^{1}$
		\\
		\bigskip
		\bf{1} Department of Mathematics, Imperial College London, UK
		\\
		\bf{2} School of Computation, Information and Technology, Department of Mathematics, Technical University of Munich, Germany
		\\
		\bigskip
		$\dagger$ Corresponding e-mail: r.talbut21@imperial.ac.uk
	\end{flushleft}
	
	
	\section*{Abstract}
	
	The problem of comparing probability distributions is at the heart of many tasks in statistics and machine learning.  Established comparison methods treat the standard setting that the distributions are supported in the same space. Recently, a new geometric solution has been proposed to address the more challenging problem of comparing measures in Euclidean spaces of differing dimensions.  Here, we study the same problem of comparing probability distributions of different dimensions in the tropical setting, which is becoming increasingly relevant in applications involving complex data structures such as phylogenetic trees.  Specifically, we construct a Wasserstein distance between measures on different tropical projective tori---the focal metric spaces in both theory and applications of tropical geometry---via tropical mappings between probability measures.  We prove equivalence of the directionality of the maps, whether mapping from a low dimensional space to a high dimensional space or vice versa. As an important practical implication, our work provides a framework for comparing probability distributions on the spaces of phylogenetic trees with different leaf sets. We demonstrate the computational feasibility of our approach using existing optimisation techniques on both simulated and real data.
	
	\paragraph{Keywords:} Simple projection; tropical matrix maps; tropical projective torus; Wasserstein metrics; phylogenetic data.
	
	
	\section{Introduction}\label{sec:intro}
	
	Some of the most immediate and important tasks in statistics and machine learning---such as clustering \citep{irpino2008dynamic} and density estimation \citep{sheather2004density}---require measuring a distance between probability distributions.  This is usually done with a notion of distance or divergence between measures; some of the most well-known distances and divergences between probability distributions are the Kullback--Leibler divergence, total variation distance, and Hellinger distance.  These notions fall within the class of maps known in probability theory as \emph{$f$-divergences} \citep{csiszar1967information}. More recently, the Wasserstein distance between measures has become increasingly relevant in statistics as it metrises weak convergence of measures and reflects the geometry of the underlying state space \citep[e.g.,][]{panaretos2019statistical, panaretos2020invitation}. 
	
	An underlying assumption behind computing $f$-divergences and Wasserstein distances, however, is that the measures live in the same space.  Nevertheless, there exist many application settings where the measures to be compared occur in distinct spaces of different dimensions. For example, phylogenetic trees are used to capture the progression of cancer in a biomedical application or the evolutionary patterns of the spread of disease in a public health application. In such applications, the numbers of leaves in the trees can differ before and after intervention such as treatment or public inoculation, and therefore such datasets will lie on spaces of different dimensions. Existing work by \cite{ren2017combinatorial} and \cite{grindstaff2019representations} study the problem of comparing sets of trees with differing numbers of leaves in BHV tree space \citep{billera2001geometry}; here, we consider probability measures rather than discrete data sets in tropical tree space \citep{speyer2004grassmannian}.
	
	There is an inherent connection between the space of phylogenetic trees and tropical geometry \citep{speyer2004grassmannian}, which has recently been exploited to develop many statistical and machine learning tools for sets of phylogenetic trees \citep[e.g.,][]{barnhill2023tml, tang2020tropical, yoshida2019tropicalPCA, aliatimis2023tropical}.  The ambient space of phylogenetic trees as well as Gr\"{o}bner complexes---central objects in tropical geometry---is the \emph{tropical projective torus}, which makes it a fundamental space in tropical geometry for both theory and applications. In this work, we therefore study Wasserstein distances between probability distributions on tropical projective tori of different dimensions.
	
	Our strategy is to leverage recent work in the Euclidean setting by \cite{cai2022distances}, who construct a Wasserstein distance between Euclidean spaces of different dimensions. The optimal transport problem and thus Wasserstein distances are well defined and have been previously studied on the tropical projective torus by \cite{lee2022tropical}.  In this work, we build on these existing results and relevant foundations to construct a Wasserstein distance between measures on tropical projective tori of different dimensions. 
	We prove that there need not be a choice of whether to map from the lower dimensional torus to the higher dimensional one or vice versa, because the directional mappings are equivalent.
	
	The remainder of this manuscript is structured as follows.  We begin with an outline of the relevant background and concepts in \Cref{sec:background}.  We then characterise the tropical equivalents to the necessary tools in \Cref{sec:tropical_tools}, which leads to our main theorem and a discussion of practical implications for phylogenetic trees in \Cref{sec:main_thm}. \Cref{sec:computation} consists of numerical experiments for the computation of our proposed distance and a demonstration of our algorithm using real influenza data. We close with a discussion of future work in \Cref{sec:discussion}.
	
	Throughout this manuscript we use the term polyhedron to refer to a (possibly unbounded) finite intersection of closed half spaces of dimension $d$, in accordance with \cite{ziegler2012lectures}.
	
	
	\section{Background and Preliminaries}
	\label{sec:background}
	
	In this section, we present the tropical setting of our work, the Wasserstein distance definition, and an overview of the approach for mapping between probability distributions of different dimensions in the Euclidean case, which we will closely follow in our work.
	
	\emph{Notation.} Throughout this manuscript, we use the notation $[n] := \{1,2,\ldots,n\}$.
	
	\subsection{Essentials of Tropical Algebra and Geometry}
	
	We begin by outlining the basic concepts of tropical algebra and tropical geometry needed for our work; a complete introduction is given by \cite{maclagan2015introduction}.
	
	\subsubsection{Tropical Algebra and Phylogenetics}
	
	\begin{definition}[Tropical Algebra]
		The \emph{tropical algebra} is the semiring $\overline\R = \R \cup \{- \infty\}$ with the addition and multiplication operators---tropical addition and tropical multiplication, respectively---given by
		\begin{align*}
			a \boxplus b = \max \{ a,b \}, \quad a \odot b = a+b. 
		\end{align*}
		The additive identity is $-\infty$ and the multiplicative identity is $0$. Tropical subtraction is not defined; tropical division is given by classical subtraction.
	\end{definition}
	
	\begin{remark}
		Note that \cite{maclagan2015introduction} use the min-plus convention where tropical addition is given by the minimum between two elements, rather than the maximum as above.  While they are equivalent, we use the max-plus convention as this has been used more frequently in recent applications  \citep[e.g.,][]{yoshida2019tropicalPCA,yoshida2022hit,aliatimis2023tropical}.
	\end{remark}
	
	Using the tropical algebra, we can define tropical parallels to most algebraic objects.  For example, we can define tropical polynomials as the maximum of finitely many linear maps with integer coefficients. 
	Evaluating functions or other mathematical expressions using the tropical algebra is referred to as \emph{tropicalisation}.
	
	The metric space in which we work is the following.
	\begin{definition}[Tropical projective torus, Tropical metric]
		\label{def:tpt}
		The $n$-dimensional \emph{tropical projective torus} is a quotient space constructed by endowing $\R^{n+1}$ with the equivalence relation 
		\begin{equation}
			\label{eq:tpt_equiv}
			\mathbf{x} \sim \mathbf{y} \Leftrightarrow \mathbf{x} = a \odot \mathbf{y};
		\end{equation}
		it is denoted by $\TPT{n+1}$.  
		The generalised Hilbert projective metric, also referred to as the \emph{tropical metric}, is given by
		\[
		\dtr(\vec x,\vec y) = \max_i \{ x_i-y_i \} - \min_i \{ x_i - y_i \} = \max_{i,j} \{ x_i - y_i - x_j + y_j\},
		\]
		and is induced by the tropical norm
		\[
		\| x\|_{\tr} = \max_i x_i - \min_i x_i.
		\]
	\end{definition}
	
	As we are mapping between tropical projective tori of different dimensions, when necessary we will distinguish between metrics on different spaces by denoting the $\TPT{n}$ metric by $\dtr^n$. \Cref{fig:trop-ball} shows the tropical ball on the tropical projective plane, where we have normalised the first coordinate to embed into $\R^2$. We use this embedding, given more generally by
	\[
	\varphi: \TPT{n+1} \rightarrow \R^{n}, \quad (x_0, \dots, x_n) \mapsto (x_2-x_0, \dots, x_n-x_0),
	\]
	for all visualisation figures in this paper. 
	
	\begin{figure}[h!]
		\centering
		\begin{tikzpicture}
			\draw [->,thick] (0,-2) -- (0,2);
			\draw [->,thick] (-2,0) -- (2,0);
			\draw [fill = gray, opacity = 1/2](0,1) -- (1,1) -- (1,0) -- (0,-1) -- (-1,-1) -- (-1, 0) -- (0,1);
			
		\end{tikzpicture}
		\caption{The tropical ball under $\varphi: \TPT{3} \rightarrow \R^2$}
		\label{fig:trop-ball}
	\end{figure}

	\subsubsection*{Phylogenetic Trees}
	
	The space of \emph{phylogenetic trees} with $N$ leaves, which are metric $N$-trees with non-negative lengths on all edges, have an important connection to tropical geometry. It was proved by  \cite{speyer2004grassmannian} that the \emph{tropical Grassmannian}---the projective variety obtained by tropicalising the Grassmannian---is homeomorphic to the space of phylogenetic trees, and therefore the tropical projective torus is the ambient space of phylogenetic trees.  Phylogenetic trees are a fundamental data type in biology that model evolutionary processes in many settings, such as the evolution of species and disease, as well as the spread of pathogens. The tropical projective torus, therefore, is an important space in mathematical biology as the ambient space of phylogenetic trees.
	
	\subsubsection{Tropical Matrices}
	
	To respect the tropical structure of the tropical projective torus, we use tropical matrices to map probability distributions from one tropical projective torus to another. Here we outline the behaviour of tropical matrix maps and their connections to tropical convexity as first presented by \cite{develin2004tropical} and discussed in Chapter 5.2 of \cite{maclagan2015introduction}. The foundational objects of convexity analysis --- line segments and convex hulls --- are defined in the tropical setting as follows:
	
	\begin{definition}[Tropical line segment, Tropical convex hull]
		\label{def:tls_tch}
		For any two points $\vec a, \vec b \in \TPT{n}$, the \emph{tropical line segment} between $\vec a$ and $\vec b$ is the set
		$$
		\gamma_{\vec a \vec b} = \{ \alpha \odot \vec a \boxplus \beta \odot \vec b \mid \alpha, \beta \in \R\}
		$$
		with tropical addition taken coordinate-wise.
		
		For a finite subset $X = \{\vec x_1,\ldots, \vec x_r\} \subset \TPT{n}$, the \emph{tropical convex hull} of $X$ is the smallest subset containing $X$ where the tropical line segment between any two points in $X$ is contained in $X$; it is the set of all tropical linear combinations of $X$,
		$$
		\mathrm{tconv}(X) = \{ \alpha_1 \odot \vec x_1 \boxplus \alpha_2 \odot \vec x_2 \boxplus \cdots \boxplus \alpha_r \odot \vec x_r \mid \alpha_1, \ldots, \alpha_r \in \R\}.
		$$
	\end{definition}
	
	We will be studying tropical matrices $M\in\overline{\R}^{m\times n}$ with at least one real entry per row---a non-degeneracy condition which ensures $M$ maps into $\TPT{m}$. This condition is assumed through the rest of the manuscript. From the definition above, we can observe that the \emph{image} of a tropical matrix $M$ in $\TPT{m}$ is the tropical span of the columns of $M$; when $M$ has all real entries, its image in $\TPT{m}$ is the tropical convex hull of its columns \citep{maclagan2015introduction}.
	
	\begin{example}
		\label{ex:bdd_matrix_im}
		Let  $M := \begin{pmatrix} 2 & 0 & 0 \\ -2 & 2 & 1 \\ 1 & 3 & -1 \end{pmatrix}.$  Its image is given in \Cref{fig:bdd_matrix_im}.
		\begin{figure}[ht!]
			\centering
			\begin{tikzpicture}[scale = 0.7]
				\draw [fill = gray] (1,2) -- (2,3) -- (-2,-1) -- (-4,-1) -- (1,-1) -- cycle;
				\node (A) at (2,3) [right] {\scriptsize$(2,3)$};
				\node (B) at (1,-1) [right] {\scriptsize$(1,-1)$};
				\node (C) at (-4,-1) [below] {\scriptsize$(-4,-1)$};
			\end{tikzpicture}
			\caption{The image of $M$ from \Cref{ex:bdd_matrix_im} under $\varphi: \TPT{3} \rightarrow \R^2$.}
			\label{fig:bdd_matrix_im}
		\end{figure}
	\end{example}
	
	The local behaviour of tropical matrices is classically linear, but to understand the global behaviour of their images and fibres, we must understand the combinatoric structure of tropical matrices. This combinatoric information is given by the \emph{type} at a point $\vec x$, and dictates the local behaviour at $\vec x$.
	
	\begin{definition}
		\label{def:type}
		Let $M\in\overline{\R}^{m\times n}$. The \emph{type} of a point $\vec x\in\TPT{n}$ with respect to $M$ is the $n$-tuple $S = (S_1,\dots,S_n)$, where
		\[
		S_j = \{ i \in [m]: (M \vec x)_i=M_{i\cdot} \odot \vec x   \text{ attains its maximum at }M_{ij}+x_j \}.
		\]
		The type of a point is denoted by $\type(\vec x)$. The set of all points with type $S$ forms a convex polyhedron which is denoted by
		\[
		X_S := \{ \vec x \in \TPT{n} : S\subseteq\type(\vec x) \}.
		\]
	\end{definition} 
	Note that the type sets $S$ cover $[m]$. This definition is given by \cite{develin2004tropical} for real-valued matrices; we note it is also valid for matrices with $-\infty$ entries.
	
	Intuitively, $S_j$ is the set of coordinates $i$ of $M\vec x$ that depend on the $j$th coordinate of $\vec x$; in neighbourhoods where $S_j = \emptyset$ then $M\vec x$ is invariant under changes to $x_j$, and in neighbourhoods where $S_j = [m]$ for some $j$ then $M\vec x$ takes the form $x_j \odot M_{\cdot j}$ and is therefore constant due to the equivalence relation on the tropical projective torus.
	
	The collection of all type cells $X_S$ form a cell decomposition of $\TPT{n}$ (Theorem 15 by \cite{develin2004tropical}) where $X_S \leq X_T$ if, and only if, $\forall\ j, \, T_j \subseteq S_j$. For any two types $S$ and $T$, $X_S\cap X_{T}=X_{S\cup T}$. \Cref{fig:matrix_map_types} shows the cell decomposition for a $2 \times 3$ matrix, and the local classically linear form of $M \vec x$ according to the type.
	
	\begin{figure}[ht!]
		\centering
		\begin{tikzpicture}[scale = 0.7,blend group = screen]
			\node (A) at (-2,-1) [right] {\scriptsize\,$(0,a-b,a-c)$};
			\node (B) at (5,1) [right] {\scriptsize\,$(0,d-e,d-f)$};
			\coordinate (A) at (-2,-1);
			\coordinate (B) at (5,1);
			\draw (A) -- (3,4);
			\draw (B) -- (8,4);
			\draw (A) -- (-6,-1);
			\draw (B) -- (-6,1);
			\draw (A) -- (-2,-3);
			\draw (B) -- (5,-3);
			\node[text width = 100] at (-3.5,-2) {\scriptsize Type: $(\{1,2\}, \emptyset, \emptyset)$ \\ $M \vec x = (a+x_1,d+x_1)$};
			\node[text width = 100] at (-3.5,0) {\scriptsize Type: $(\{2\}, \emptyset, \{1\})$ \\ $M \vec x = (c+x_3,d+x_1)$};
			\node[text width = 100] at (-2,2) {\scriptsize Type: $(\emptyset, \emptyset, \{1,2\})$ \\ $M \vec x = (c+x_3,f+x_3)$};
			\node[text width = 100] at (4,2) {\scriptsize Type: $(\emptyset, \{1\}, \{2\})$ \\ $M \vec x = (b+x_2,f+x_3)$};
			\node[text width = 100] at (7.5,-0.5) {\scriptsize Type: $(\emptyset, \{1,2\}, \emptyset)$ \\ $M \vec x = (b+x_2,e+x_2)$};
			\node[text width = 100] at (3,-2) {\scriptsize Type: $(\{2\}, \{1\}, \emptyset)$ \\ $M \vec x = (b+x_2,d+x_1)$};
			\fill[black] (A) circle (2pt);
			\fill[black] (B) circle (2pt);
		\end{tikzpicture}
		\caption{\centering$\TPT{3} \cong \R^2$ partitioned by type with respect to $M:= \begin{pmatrix} a & b & c\\ d & e & f \end{pmatrix}$.}
		\label{fig:matrix_map_types}
	\end{figure}
	
	\subsection{Wasserstein Distances Between Probability Distributions}
	\label{sec:wass}
	
	\emph{Wasserstein distances} measure distances between probability distributions.  They arise as a certain class of optimal transport problem, with applications in PDEs, image processing, optimisation and statistics. As opposed to $f$-divergences, they preserve the geometry of the state space over which they are defined and metrise weak convergence results such as the central limit theorem \citep[e.g.,][Section 3.1]{lee2022tropical}.
	
	The Wasserstein distance is a solution to a particular case of the \emph{optimal transport problem}, first introduced by \cite{monge1781memoire}.  The optimal transport problem is the search for an optimal mapping to transport a set of resources from their sources to sinks which minimises the transport cost.  When the cost function is given by a distance between source and sink, the solution to the optimal transport problem yields the Wasserstein distances. The space on which the source and sink locations are located is referred to as the \emph{state space}, while the physical distance that the resources need to be transported is measured by the \emph{ground metric}.
	
	The optimal transport problem was then relaxed to a probabilistic framework \citep{kantorovich1942translocation}. The question is then to find the optimal coupling of two random variables which minimises the expectation of a cost function; when the cost function is the ground metric of the state space, the solution to the problem gives the Wasserstein distance between probability measures.
	
	\begin{definition}[$p$-Wasserstein Distance]
		Let $(\Omega,d)$ be a Polish metric space and $p \in [1, \infty)$. Let $\mu, \nu \in P(\Omega)$ and let $\Pi(\mu,\nu)$ be the set of all couplings of $\mu$ and $\nu$. The \emph{$p$-Wasserstein distance} $W_p$ on $P(\Omega)$ is defined by
		$$
		\label{eqn:wass}
		W_p(\mu,\nu)^p \coloneqq \inf_{\pi \in \Pi(\mu,\nu)} \E_{(X,Y) \sim \pi} \left[ d(X,Y)^p \right].
		$$
	\end{definition}
	
	The $p$-Wasserstein distance is not necessarily finite on all of $P(\Omega)$, but will be finite on measures with finite $p$-moments, which we denote by $P_p(\Omega)$ \citep{villani2009optimal, villani2003topics}.
	
	The tropical projective torus is a Polish space \citep{monod2018StatPerspective}, while the \emph{tropical $p$-Wasserstein distance} is well-defined and has been previously studied by \cite{lee2022tropical} using the tropical projective torus as the state space and the tropical metric as the ground metric.
	
	\subsection{Wasserstein Distances Between Euclidean Spaces of Different Dimensions}
	\label{sec:cai_lim}
	
	To measure differences between probability distributions supported on Euclidean space of different dimensions, \cite{cai2022distances} constructed a Wasserstein pseudo-metric; it is a pseudo-metric in that it is symmetric and positive up to isometries. They note this pseudo-metric reflects well-understood properties of certain families of measures, as well as demonstrating computational advantages over the Gromov--Wasserstein distance \citep{memoli2011gromov}. We look to mirror their approach in the more complex tropical setting. \cite{cai2022distances} also establish the same results for $f$-divergences, but we will restrict our considerations to Wasserstein distances due to their inherently geometric behaviour; this behaviour is favourable for applications where our data has a tropical structure such as phylogenetic tree data.
	
	We will leverage the approach of \cite{cai2022distances} in our work and so we provide an outline of their strategy here. Suppose we have two measures, $\mu, \nu$ in $P_p(\R^m)$ and $P_p(\R^n)$ respectively, with $m \leq n$; we suppose there is no known correspondence between the coordinate systems of $\R^m$ and $\R^n$, even if $m=n$.
	
	Both measures are required to be on the same state space to take a Wasserstein distance, so \cite{cai2022distances} define their projections of interest as those composed of a semi-orthogonal matrix and a translation:
	$$
	\mathcal{M} = \{ \phi: \R^n \rightarrow \R^m : \phi(x)=A\vec x+b, \, A \in O(m,n), \, b \in \R^m \}.
	$$
	
	Then the sets of projection and embedding measures are defined as, respectively,
	\begin{gather*}
		\Phi^-(\nu,m) \coloneqq \{ \beta \in P(\R^m) : \beta = \phi_*(\nu) \text{ for some } \phi \in \mathcal{M}  \}, \label{eq:ProjMeasuresEuclid}\\
		\Phi^+(\mu,n) \coloneqq \{ \alpha \in P(\R^n) : \mu = \phi_*(\alpha) \text{ for some } \phi \in \mathcal{M}  \}, \label{eq:EmbdMeasuresEuclid}
	\end{gather*}
	where $\phi_*(\alpha)$ denotes the push-forward of $\alpha$; that is, $\phi_*(\alpha)(A) = \alpha(\phi^{-1}(A))$. 
	
	From these projected and embedded measures, the projection and embedding Wasserstein distances are defined respectively by
	\begin{align}
		W_p^-(\mu,\nu) = \inf_{\beta \in \Phi^-(\nu,m)} W_p(\mu, \beta), \label{eq:ProjWDist} \\
		W_p^+(\mu,\nu) = \inf_{\alpha \in \Phi^+(\mu,n)} W_p(\alpha, \nu). \label{eq:EmbdWDist}
	\end{align}
	
	While both the projection and embedding distances offer an intuitive measure for comparing $\mu$ and $\nu$, it is not obvious which would be more meaningful in practice: to map from the lower dimensional space to the higher dimensional one, or vice versa. The main result by \cite{cai2022distances} tells us we need not make an arbitrary choice; the two distances are equal.
	
	\begin{theorem}{\cite[][Theorem 2.2]{cai2022distances}}
		\label{thm:cai_lim}
		For $m \leq n, p \in [1, \infty]$, let $\mu \in P_p(\R^m), \nu \in P_p(\R^n)$. Then
		\[
		W^-(\mu,\nu) = W^+(\mu,\nu)
		\]
		where $W^-$ and $W^+$ are defined by \eqref{eq:ProjWDist} and \eqref{eq:EmbdWDist}.
	\end{theorem}
	
	The focal contribution of our work is a tropical parallel to the framework of \cite{cai2022distances}, allowing us to map between probability measures in tropical projective tori of different dimensions and establish the same metric equivalence in the more complex tropical setting.

	\section{Building Tropical Tools}
	\label{sec:tropical_tools}
	
	To build up to the tropical equivalent of \Cref{thm:cai_lim}, we first require tropical counterparts of central objects to the construction by \cite{cai2022distances}.  Specifically, we require a semi-orthogonal map in the tropical setting.  Notice that in our setting, the equivalence relation \eqref{eq:tpt_equiv} defining the tropical projective torus must always be preserved in mappings.
	
	\subsection{Tropical Matrix Maps}
	
	Standard linear maps fail to preserve the equivalence relation \eqref{eq:tpt_equiv} that characterizes the tropical projective torus, making them unsuitable for our aim of mapping between tropical projective tori. Instead, we consider tropical linear maps. We now review the properties of general tropical matrices, highlighting their irregular properties; namely, that tropical matrices can have bounded image and fibres of different dimension. 
	
	\subsubsection{Metric Geometry Under Tropical Linear Maps.}
	
	To formulate Wasserstein distances over tropical linear maps, we must understand their metric geometry. Here we prove a useful property of tropical linear maps, namely that they are always non-expansive with respect to the tropical metric.
	
	\begin{proposition}
		\label{prop:Non-Expansive_Map}
		For any tropical matrix $M$:
		\[
		\forall\ \vec x, \vec y \in \TPT{n}: \quad \dtr(M \vec x,M \vec y) \leq \dtr( \vec x, \vec y)
		\]
	\end{proposition}
	
	\begin{proof}
		There exist coordinates $r,s$ such that 
		\[
		\dtr(M\vec x, M\vec y) = \max_i \{M_{ri} + x_{i}\} - \max_j \{M_{rj} + y_{j}\} - \max_{k}\{M_{sk} + x_{k}\} + \max_{\ell}\{M_{s\ell} + y_{\ell} \},
		\]
		Fix $i,j,k,\ell$ as the maximal arguments for the terms above.
		Then
		\begin{align*}
			\dtr(M \vec x,M \vec y) &= M_{ri} + x_{i} - (M_{rj} + y_{j}) - (M_{sk} + x_{k}) + M_{s\ell} + y_{\ell} \\
			&\leq M_{ri} + x_{i} - (M_{ri} + y_{i}) - (M_{s\ell} + x_{\ell}) + M_{s\ell} + y_{\ell} \\
			&\leq x_{i} - y_{i} - x_{\ell} + y_{\ell} \leq \dtr(\vec x, \vec y).
		\end{align*}
	\end{proof}
	
	In light of \Cref{prop:Non-Expansive_Map}, we note the non-expansive property is more restrictive in the Euclidean setting than the tropical setting; every tropical matrix map is non-expansive, but not every Euclidean matrix map is non-expansive.
	
	\subsubsection{Images of Tropical Matrices.} To define embedded measures, we must take a pull-back through $M$; this requires the image of $M$ to cover the support of our measure $\mu \in P_p(\TPT{m})$. However, tropical matrices often have bounded image, as we saw in \Cref{ex:bdd_matrix_im}. The following result characterises surjectivity for tropical matrix maps.
	
	\begin{lemma}\label{lem:surj}
		A tropical matrix map $M$ is surjective on $\TPT{m}$ if and only if, for each row $i$, there is at least one column $c(i)$ such that
		\begin{equation}
			\label{eqn:surj}
			M_{r,c(i)}\in\begin{cases}
				-\infty\qquad &\emph{if}\,\,r\neq i,\\
				\R\qquad &\emph{if}\,\,r = i.
			\end{cases}
		\end{equation}
	\end{lemma}
	
	\begin{proof}
		Suppose $M$ is surjective. Then let $\vec x$ be such that $M\vec x = [(0, \dots, K, \dots, 0)]$, where only the $i$-th entry is $K$, and $K$ is a positive real number such that
		$$
		K>\max\{M_{i_0j_0}-M_{i_1j_1}: M_{i_0j_0},M_{i_1j_1} \in \R\}.
		$$  
		Let $j \in [n]$ be some coordinate such that $(M\vec x)_i = M_{ij} + x_j = K$. If $M_{lj}$ is real for any other coordinate, then 
		$$
		0 = (M \vec x)_l \geq M_{lj} + x_j = M_{lj}-M_{ij}+M_{ij}+ x_j = M_{lj} - M_{ij} + K > 0,
		$$
		a contradiction. Notice that the final equality comes from the definition of $\vec x$, while the last equality is a consequence of the definition of $K$.
		
		To prove the other direction, let us consider $\vec y=(y_1,\ldots,y_m)\in\TPT{m}$. For each $i\in[m]$, we can pick some column $c(i)$ that satisfies \eqref{eqn:surj}. Let $I := \{c(i) : i \in [m] \}$. Then we define $K$ as before and define $\vec x=(x_1,\dots,x_n)\in\TPT{n}$ by
		$$
		x_j:=\begin{cases}
			y_{i}-M_{i,c(i)}\qquad&\emph{if}\,\,j=c(i),\\
			\min_{i \in [m] } \{ -K-y_{i}+M_{i,c(i)} \} \qquad&j \notin I
		\end{cases}
		$$
		By direct computation, we can check that the maximum for each row $i$ is achieved at $c(i)$ and that the value is exactly $y_i$.
	\end{proof}
	
	\subsubsection{Fibres of Tropical Matrices.}
	The semi-orthogonal projections used by \cite{cai2022distances} are well-behaved in the sense that they preserve some of the structure of $\nu \in P_p(\R^n)$. For example, a push-forward through a semi-orthogonal projection cannot increase the co-dimension of a measure \citep{heurteaux2007dimension}, as each fibre has dimension $n-m$. It is the \emph{fibres} of $M$ which determine the dimensionality of a projected measure and are therefore central to our work.
	\begin{definition}
		The \emph{fibre} of $M$ at a point $\vec y \in \TPT{m}$ is given by $F_{\vec y} = \{ \vec x\in\TPT{n} : M \vec x = \vec y \}$.
	\end{definition}
	
	\begin{example}
		\label{ex:fibres}
		Let $M:= \begin{pmatrix} a & b & c\\ d & e & f \end{pmatrix}$, as in \Cref{fig:matrix_map_types}. If $S_j$ is empty at $\vec x$, then $M\vec x$ is locally independent of $e_j$; from this we can draw the fibres of $M$, as illustrated in \Cref{fig:matrix_map_fibres} --- points of the same colour are mapped to the same point by $M$.
		\begin{figure}[ht!]
			\centering
			\begin{tikzpicture}[scale = 0.7,blend group = screen]
				\draw [fill=red, opacity=1/4] (-6,-3) -- (-2,-3) -- (-2,-1) -- (-6,-1);
				\draw [fill=red!71.4!blue, opacity=1/4] (-6,1) -- (-6,4) -- (3,4) -- (0,1);
				\draw [fill=blue, opacity=1/4] (5,1) -- (8,4) -- (9,4) -- (9,-3) -- (5,-3);
				\node (A) at (-2,-1) [right] {\scriptsize\,$(0,a-b,a-c)$};
				\node (B) at (5,1) [right] {\scriptsize\,$(0,d-e,d-f)$};
				\coordinate (A) at (-2,-1);
				\coordinate (B) at (5,1);
				\foreach \i in {1,...,5} {
					\pgfmathtruncatemacro{\tmp}{100-4.77*\i}
					\draw[red!\tmp!blue, opacity = 1/2] (-6,\i/3-1) -- (\i/3-2,\i/3-1) -- (\i/3-2,-3);
				}
				\draw [red!71.4!blue, opacity = 1/2] (0,1) -- (0,-3);
				\foreach \i in {1,...,14} {
					\pgfmathtruncatemacro{\tmp}{71.4 - 4.77 * \i}
					\draw[red!\tmp!blue, opacity = 1/2] (\i/3+3,4) -- (\i/3,1) -- (\i/3,-3);
				}
				\draw (A) -- (3,4);
				\draw (B) -- (8,4);
				\draw (A) -- (-6,-1);
				\draw (B) -- (-6,1);
				\draw (A) -- (-2,-3);
				\draw (B) -- (5,-3);
				\node[text width = 100] at (-3.5,-2) {\scriptsize Type: $(\{1,2\}, \emptyset, \emptyset)$ \\ $Mx = (a+x_1,d+x_1)$};
				\node[text width = 100] at (-3.5,0) {\scriptsize Type: $(\{2\}, \emptyset, \{1\})$ \\ $Mx = (c+x_3,d+x_1)$};
				\node[text width = 100] at (-2,2) {\scriptsize Type: $(\emptyset, \emptyset, \{1,2\})$ \\ $Mx = (c+x_3,f+x_3)$};
				\node[text width = 100] at (4,2) {\scriptsize Type: $(\emptyset, \{1\}, \{2\})$ \\ $Mx = (b+x_2,f+x_3)$};
				\node[text width = 100] at (7.5,-0.5) {\scriptsize Type: $(\emptyset, \{1,2\}, \emptyset)$ \\ $Mx = (b+x_2,e+x_2)$};
				\node[text width = 100] at (3,-2) {\scriptsize Type: $(\{2\}, \{1\}, \emptyset)$ \\ $Mx = (b+x_2,d+x_1)$};
				\fill[black] (A) circle (2pt);
				\fill[black] (B) circle (2pt);
			\end{tikzpicture}
			\caption{\centering$\TPT{3} \cong \R^2$ partitioned by type with respect to $M$ in \Cref{ex:fibres}.}
			\label{fig:matrix_map_fibres}
		\end{figure}
	\end{example}
	
	The following lemma characterises the fibres of general tropical matrices as polyhedral complexes; this can be seen as a refinement of Theorem 15 by \cite{develin2004tropical}.
	
	\begin{lemma}
		The fibres $F_{\vec y} = \{\vec x \in \TPT{n}: M(\vec x)=\vec y \}$ of a general tropical matrix $M$ are (unbounded) polyhedral complexes.
		\label{lem:fiber_complex}
	\end{lemma}
	
	\begin{proof}
		
		For a given type $S$, we define
		$$
		B^S_{\vec y} := \bigcap_{i \in [m]} \{\vec x \in \TPT{n} : x_{j} - x_{k} = M_{1k}-M_{ij}+y_i-y_1\mid 1\in S_{k}, i\in S_{j} \}. 
		$$
		
		Then we define $C^S_{\vec y}:=X_S\cap B^{S}_{\vec y}$ and let $\mathcal{C}_{\vec y}:=\{C^S_{\vec y} : S \text{ is a type}\}$. We will show that $\mathcal{C}_{\vec y}$ is a polyhedral complex in $\TPT{n}$. Since for every type $S$, both $X_S$ and $B^{S}_{\vec y}$ are polyhedra in $\TPT{n}$, $C^S_{\vec y}$ is a well-defined polyhedron in $\TPT{n}$ as well.\\
		
		We first prove $\mathcal{C}_{\vec y}$ is closed under intersections. Consider two types $S$ and $T$; we can write $C^S_{\vec y}\cap C^{T}_{\vec y}$ as
		$$
		(X_S\cap X_T)\cap(B^S_{\vec y}\cap B^{T}_{\vec y})=X_{S\cup T}\cap(B^S_{\vec y}\cap B^{T}_{\vec y})
		$$
		and we will prove this last term is equal to $X_{S\cup T}\cap B^{S\cup T}_{\vec y}$. By definition, $B^{S\cup T}_{\vec y}$ is equal to:
		$$
		\bigcap_{i \in [m]} \{\vec x \in \TPT{n} : x_{j} - x_{k} = M_{1k}-M_{ij}+y_i-y_1\mid 1\in S_{k}\cup T_{k}, i\in S_{j}\cup T_{j} \}.
		$$
		We see that $B^{S\cup T}_{\vec y} \subset B^{S}_{\vec y} \cap B^{T}_{\vec y}$, so it remains to show that for any $\vec x\in X_{S\cup T}\cap(B^S_{\vec y}\cap B^{T}_{\vec y})$, $\vec x$ belongs to $B^{S\cup T}_{\vec y}$ as well. This amounts to proving that: 
		$$
		x_{j} - x_{k} = M_{1k}-M_{ij}+y_i-y_1,\qquad \forall i\in[m], 1\in S_{k}\cup T_{k}, i\in S_{j}\cup T_{j}.
		$$
		Let $i,j,k$ be such that $1\in S_k$ and $i\in S_j$. In this case, the equation holds trivially based on the definition of $B^{S}_{\vec y}$. Similarly, if $1\in T_k$ and $i\in T_j$ we are done. So, the remaining case is the one in which $1\in S_k$ and $i\in T_j$. We consider $j_0$ such that $i\in S_{j_0}$, notice that this exists since $S$ covers $[m]$. Since $\vec x\in X_{S\cup T}$ the following equation holds:
		$$
		M_{ij}+x_j=M_{ij_0}+x_{j_0}=y_i,
		$$
		implying that $x_j=M_{ij_0}-M_{ij}+x_{j_0}$. Now we can write $x_j-x_k$ in the following form:
		\begin{align*}
			x_j-x_k &= M_{ij_0}-M_{ij}+x_{j_0}-x_k \\
			&=M_{ij_0}-M_{ij}+M_{1k}-M_{ij_0}+y_i-y_1 \\
			&=M_{1k}-M_{ij}+y_i-y_1,
		\end{align*}
		where for the first equality we used that $\vec x\in B^S_{\vec y}$. Notice that we don't need to prove the case in which $1\in T_k$ and $i\in S_i$, as the roles of $S$ and $T$ are equivalent and so the same proof applies swapping their roles. So we proved that $\mathcal{C}_{\vec y}$ is closed under intersection.\\ 
		
		We prove now that the faces of elements of $\mathcal{C}_{\vec y}$ are contained in $\mathcal{C}_{\vec y}$. Since by definition $B^S_{\vec y}$ is an affine space for any type $S$, the faces of $C^S_{\vec y}$ are of the form $X_T \cap B^S_{\vec y}$ where $X_T$ is one of the faces of $X_S$; that is, $S \leq T$. It suffices to show now that $X_T \cap B^S_{\vec y} = X_T \cap B^T_{\vec y}$.
		
		As $S \leq T$, it is immediate that $B_{\vec y}^T \subset B_{\vec y}^S$. It remains to show that for all $\vec x \in X_T\cap B_{\vec y}^S$ we have $\vec x \in B_{\vec y}^T$.
		
		Consider $i,j,k$ such that $1 \in T_k, i \in T_j$; we will show that $x_j - x_k = M_{1k}-M_{ij}+y_i-y_1$. As $S$ covers $[m]$ and we have $S \leq T$, there are some $k', j'$ such that $1 \in S_{k'}, T_{k'}$ and $i \in S_{j'}, T_{j'}$. For all $\vec x \in X_T$ we have
		\[
		x_k+M_{1k} = x_{k'}+M_{1k'}, \qquad x_j+M_{ij} = x_{j'}+M_{ij'}.
		\]
		Hence, for all $\vec x \in X_T\cap B_{\vec y}^S$:
		\begin{align*}
			x_j - x_k &= x_{j'} + M_{ij'} - M_{ij} - x_{k'} - M_{1k'} + M_{1k}, \\
			&= y_i - M_{ij}  - y_1 + M_{1k},
		\end{align*}
		where the first equality uses that $\vec x \in X_T$ and the second equality uses that $\vec x \in B_{\vec y}^S$. We conclude that $\vec x \in B_{\vec y}^T$, and hence that $\mathcal{C}_{\vec y}$ is closed under taking faces.
		\\
		
		In order to complete the proof, it only remains to show that $F_{\vec y}=\cup_{S} C^S_{\vec y}$. 
		We know that $\TPT{n}=\cup_S X_S$, which implies that $F_{\vec y}=\cup_S(X_S\cap F_{\vec y})$. So it is enough to prove that $X_S\cap F_{\vec y}=X_S\cap B^{S}_{\vec y}$.
		
		We start by proving $X_S\cap F_{\vec y}\subseteq X_S\cap B^{S}_{\vec y}$. 
		Consider $i\in[n]$ and $k,j\in[m]$, such that $1\in S_k$ and $i\in S_j$, and fix $\vec x \in X_S\cap F_{\vec y}$. It follows from the definition of types that: 
		\begin{equation}
			\label{eq:type_condition}
			(M \vec x)_i - (M \vec x)_1 = x_j + M_{ij} - x_k - M_{1k}.
		\end{equation}
		We can therefore write $x_j-x_k$ as
		\begin{align}
			x_j-x_k&=M_{1k}-M_{ij}+(x_j+M_{ij})-(M_{1k}+x_k) \nonumber\\
			&=M_{1k}-M_{ij}+(M\vec x)_i-(M\vec x)_1 \label{eq:2nd_eq}\\
			&=M_{1k}-M_{ij}+y_i-y_1, \nonumber
		\end{align}
		where the penultimate equality is given by \Cref{eq:type_condition} and the final equality follows from $\vec x \in F_{\vec y}$. Hence $\vec x \in B_{\vec y}^S$, so we have $X_S\cap F_{\vec y}\subseteq X_S\cap B^{S}_{\vec y}$.
		
		Finally, we show $X_S \cap B_{\vec y}^S \subseteq F_{\vec y}$ by direct computation. For $i,j,k$ such that $1 \in S_k$ and $i \in S_j$:
		$$
		(M \vec x)_i - (M \vec x)_1 = M_{ij} + x_j - M_{1k} - x_k = y_i - y_1,
		$$
		where the first equality follows from $\vec x \in X_S$ and the second equality follows from $\vec x \in B_{\vec y}^S$. Therefore the fibre $F_{\vec y} = \cup_S C^S_{\vec y}$ is a polyhedral complex, as desired.
		
	\end{proof}
	
	\begin{remark}\label{rmk:maximal_cells}
		From our proof of \Cref{lem:fiber_complex}, we note that the poset of fibre cells $C^S_{\vec y}$ is given by a reverse type inclusion as in the case of type cells \citep{develin2004tropical}. Therefore, the maximal cells $C^S_{\vec y}$ of the fibre are exactly those whose type is given by a partition of $[m]$ (and copies of the empty set).
	\end{remark}
	
	We see in \Cref{fig:matrix_map_fibres} that fibre polyhedral complexes can have mixed dimension. A push-forward $M(\nu)$ through such matrices can preserve little to none of the structure of $\nu$ due to this irregular fibre dimension. We would therefore like to restrict our considerations to matrices with fibres of consistent dimension $n-m$. In the next section we identify the set of such matrices as the set of \emph{simple projections}.
	
	\subsection{Simple Projections} \label{sec:simple_projections}
	
	In addition to the assumption that the tropical matrices we study have at least one real entry per row as specified above, the following class of tropical matrices---\emph{simple projections}---have the necessary properties for our task of constructing our tropical Wasserstein distance. 
	
	\begin{definition}[Simple Projection]
		A tropical matrix $M \in \overline{\R}^{m \times n}$, where $n>m$, is called a \emph{simple projection} when each column has at most one real entry. 
		
		For $M$, we define $J_i \coloneqq \{ j: M_{ij} \in \R \}$. These sets are disjoint and non-empty, but do not necessarily cover $[n]$. Then
		\begin{align}\label{eq:simple_projections}
			(M \vec x)_i = \max_{j \in J_i} \{ M_{ij} + x_j \}.
		\end{align}
	\end{definition}
	
	The following relationship establishes when a tropical matrix is a simple projection in terms of the dimensionality of its fibres.
	
	\begin{lemma}\label{lemma:consistent_fibres} 
		$M$ is a simple projection if and only if, for all $\vec y \in \TPT{m}$, every maximal cell of the fibre $F_{\vec y}$ has dimension $n-m$.
	\end{lemma}
	
	\begin{proof}
		We will show that the following are equivalent:
		\begin{enumerate}
			\item $M$ is a simple projection, \label{TFAE:simple_proj}
			\item for all $\vec y$, the maximal cells of the fibre $F_{\vec y}$ have dimension $n-m$,\label{TFAE:fibre_dim}
			\item the maximal cells $X_S$ of the type cell decomposition of $\TPT{n}$ correspond to types consisting of $m$ singletons and $n-m$ empty sets.\label{TFAE:singleton_types}
		\end{enumerate}
		We first show that \ref{TFAE:fibre_dim} $\Leftrightarrow$ \ref{TFAE:singleton_types}.
		By \Cref{rmk:maximal_cells} a fibre cell $C^S_{\vec y}$ is maximal iff $S$ is a partition which is equivalent to $X_S$ being a maximal cell. For such an $S$, we define the classical matrix $M^S\in\R^{m\times n}$ and $b^S\in\R^m$ as follows
		\begin{equation*}
			M^S_{i,j}:=\begin{cases}
				1, &\emph{if}\,\, i\in S_j,\\
				0, &\emph{if}\,\, i\notin S_j
			\end{cases}\qquad 
			b_i:=M_{i,j(i)},
		\end{equation*}
		where $j(i)$ is the only $j$ such that $i\in S_j$. From this, we directly notice that the restriction of $M$ to $X_S$ is given by the classically affine operator mapping $x\in\R^n$ to $M^S\cdot x+b^S\in\R^m$, where here, matrix--vector multiplication is computed in the usual non-tropical manner. The dimension the maximal fibre cell $C^S_{\vec y}$ in $X_S$ is then given by the dimension of the kernel of $M^S$. Since this is a $0/1$ matrix with one non-zero entry per row, we may use a standard linear algebra argument to derive that 
		$$
		\ker M^S=|\{j\in[m] :  M^S_{\cdot,j}=\vec 0\}|=|\{j\in[m] : S_j=\emptyset \}|\leq n-m.
		$$ 
		From this we can deduce that every maximal type cell has exactly $m$ non-empty $S_j$ if and only if, $\forall\ \vec y \in \TPT{m}$, every maximal cell in $F_{\vec y}$ has dimension $n-m$. As $S$ is a partition, we conclude that \ref{TFAE:fibre_dim} $\Leftrightarrow$ \ref{TFAE:singleton_types}.
		
		We next show that \ref{TFAE:simple_proj} $\Leftrightarrow$ \ref{TFAE:singleton_types}. Suppose for contradiction that $M$ is not a simple projection, but for every maximal type cell $X_S$, the type $S$ consists of singletons. Then for some $j \in [n]$ and $i_0 \neq i_1 \in [m]$, both $M_{i_0j}$ and $M_{i_1j}$ are real. We set
		\begin{align*}
			B &:= \min \{ \min(M_{i_0j} - M_{i_0k}, M_{i_1j} - M_{i_1k}) : k \in[n]\} \\
			U &:= \{ \vec x\in\TPT{n}: x_k - x_j \leq -B\text{ for } k\in[n]\setminus\{j\}\}.
		\end{align*}
		$U$ is a full dimensional polyhedron in $\TPT{n}$, so there is some maximal $X_S$ such that $\dim(X_S\cap U)=\dim(X_S)$. By construction, for all $\vec x \in U$ and all $k \in [n]$,
		\begin{align*}
			M_{i_0k} + x_k \leq M_{i_0j} + x_j, \\
			M_{i_1k} + x_k \leq M_{i_1j} + x_j,
		\end{align*}
		and hence $i_0,i_1 \in S_{j}$ on $U_j$. Contradiction. \\
		To show the converse, suppose that $M$ is a simple projection. Then for all $\vec x$ in a maximal type cell $X_S$, $i \in S_j$ for some $j \in J_i$. As the $J_i$ are disjoint, each $S_j$ is either $\emptyset$ or a singleton.
	\end{proof}
	
	We can verify that simple projections satisfy the condition of \Cref{lem:surj} by their definition, and hence are surjective. By \Cref{lemma:consistent_fibres}, they also have fibres of dimension $n-m$. We conclude that simple projections are the most general tropical matrices satisfying these properties. 
	In \Cref{lem:def:hom} and \Cref{prop:homeomorphism} we now use these properties of simple projections to construct a homeomorphism $f_M$ between $\TPT{n}$ and $\TPT{m} \times F_{\vec 0}$ which is crucial for our main theorem.
	\begin{lemma}
		\label{lem:def:hom}
		Let $M$ be a simple projection from $\TPT{n}$ to $\TPT{m}$ and let $F_{\vec 0}$ be the fibre of $M$ at $\vec 0 \in \TPT{m}$. We define a map $f_M: \TPT{n} \rightarrow \TPT{m} \times F_{\vec 0}$ by mapping $\vec x\in\TPT{n}$ to $(M\vec x,\vec x-\vec z^{\vec x})\in\TPT{m}\times\TPT{n}$, where $\vec{z}^\vec x$ is defined as follows:
		\begin{align}
			\label{eq:def_z}
			z^{\vec x}_j := \begin{cases}
				(M\vec x)_{i} - \max_k (M\vec x)_k   & j \in J_i, \\
				0   & j \notin \cup_{i\in[n]} J_i.
			\end{cases}
		\end{align}
		The map $f_M$ is a well-defined map.
	\end{lemma}
	\begin{proof}
		We must prove that this map respects the equivalence class defining $\TPT{n}$. This amounts to proving that for all $c\in\R$ and all $\vec x\in\TPT{n}$, $f_M(\vec x)\sim f_M(c\odot\vec x)$. Let us start by considering $\vec z^{c\odot\vec x}$, which, from \eqref{eq:def_z}, is given by
		\begin{align*}
			z^{c \odot \vec x}_j &= \begin{cases}
				(M(c\odot \vec x)_{i}-\max_k(M(c\odot \vec x)_k)=M(\vec x)_{i}-\max_k(M\vec x)_k   & j \in J_i \\
				0 \qquad\qquad   & j \notin \cup_i J_i, 
			\end{cases}
		\end{align*}
		so it is equal to $\vec z^{\vec x}$. Therefore we can write $f_M(c\odot\vec x)$ as
		\begin{align*}
			(c \odot M\vec x, c \odot \vec x -\vec z^\vec x)=(c \odot M\vec x, c \odot( \vec x -\vec z^\vec x))\sim (M\vec x, \vec x- \vec z^{\vec x}) = f_M(\vec x),
		\end{align*}
		therefore proving that $f_M$ is well defined.
		
		We now prove that $\Im f_M \subset \TPT{m} \times F_{\vec 0}$. It suffices to show $\vec x-\vec z^\vec x \in F_{\vec 0}$, which can be done by writing $M(\vec x-\vec z^\vec x)_i$ for each $i\in[m]$ as 
		\begin{align*}
			\max_{j \in J_i} \{ M_{ij} + x_j - (M\vec x)_{i} - \max_k (M\vec x)_k \} &= \max_{j \in J_i} \{ x_j +M_{ij}\}  - (M\vec x)_{i} - \max_k (M\vec x)_k \\
			&=- \max_k (M\vec x)_k, 
		\end{align*}
		which is independent of $i$, and hence $M(\vec x-\vec z^\vec x)\sim\vec 0$.  Notice that for the first equality, we may take $(M\vec x)_{i}$ and $\max_k (M\vec x)_k$ out of the maximum since they are independent of $j$, while for the second equality we use the simple projection definition.
	\end{proof}
	
	\begin{proposition}   
		\label{prop:homeomorphism}
		The map $f_M$ defined in \Cref{lem:def:hom} is a homeomorphism between $\TPT{n}$ and $\TPT{m}\times F_{\vec 0}$ satisfying
		\begin{gather}
			\frac{1}{3} d_{\TPT{m}\times F_{\vec 0}}(f(\vec x^1), f(\vec x^2)) \leq \dtr^{n}(\vec x^1,\,\vec x^2) \leq d_{\TPT{m}\times F_{\vec 0}}(f(\vec x^1), f(\vec x^2)),\label{eq:metric_equivalence}
		\end{gather}
		where $d_{\TPT{m}\times F_{\vec 0}}((\vec u^1, \vec v^1), (\vec u^2, \vec v^2)) = d_{\tr}^m(\vec u^1, \vec u^2) + d_{\tr}^n(\vec v^1, \vec v^2)$.
	\end{proposition}
	
	\begin{proof}

		We begin by showing the surjectivity of $f_M$. Consider $(\vec y,\vec u)$ in $\TPT{m} \times F_{\vec 0}$ and define
		$$
		w_j := \begin{cases}
			y_{i} - \max_k y_k  & j \in J_i \\
			0   & j \notin \cup_i J_i. 
		\end{cases}
		$$
		We will show $f_M(\vec w+\vec u) \sim (\vec y,\vec u)$. Let us start by considering the entries of $M(\vec w+\vec u)$, which we can write in the following way:
		\begin{align*}
			M(\vec w+\vec u)_i &= \max_{j \in J_i} \{M_{ij} + y_{i} - \max_k y_k+u_j \} \\
			&= \max_{j \in J_i} \{M_{ij}+ u_j \} + y_i - \max_k y_k\\
			&= y_i - \max_k y_k,
		\end{align*}
		implying that $M(\vec w+\vec u)=\vec y$ and $\vec w =\vec  z^{\vec w+\vec u}$. This then implies that: $$f_M(\vec w+\vec u) = (\vec y, \vec{w+u-w}) = (\vec y,\vec u),$$ hence $f_M$ is surjective. 
		
		In order to show injectivity, consider $\vec x^1,\vec x^2\in\TPT{n}$ such that $f_M(\vec x^1)\sim f_M(\vec x^2)$. This means that there are $c_1,c_2\in\R$ where:
		\begin{align}
			M\vec x^2 &= c_1 \odot M\vec x^1, \label{eq:injectivity1}\\
			\vec x^1 - \vec z^{\vec x^1} &= c_2\odot (\vec x^2 - \vec z^{\vec x^2}). \label{eq:injectivity2}
		\end{align}
		
		From \eqref{eq:def_z} we see that \eqref{eq:injectivity1} implies 
		$\vec z^{\vec x^1} = \vec z^{\vec x^2}.$
		We can rewrite \eqref{eq:injectivity2} as
		$$
		\forall\ j\in[n]: \qquad x^1_j=z^{\vec x^1}_j+c_2+x^2_j-z^{\vec x^2}_j=z^{\vec x^2}_j+c_2+x^2_j-z^{\vec x^2}_j=c_2+x^2_j
		$$ 
		from which we deduce $\vec x^1\sim\vec x^2$, proving $f_M$ is a bijection.
		
		Finally, we prove that $f_M$ is an homeomorphism, i.e., $f_M$ and $f_M^{-1}$ are both continuous, by proving the metric equivalence \eqref{eq:metric_equivalence}.
		First, note that
			$$
			z^{\vec x^1}_j-z^{\vec x^2}_j = \begin{cases}
				(M\vec x^1)_{i} - \max_k (M\vec x^1)_k - (M\vec x^2)_{i} + \max_k (M\vec x^2)_k   &j \in J_i \\
				0   &  j \notin \cup_i J_i,
			\end{cases}
			$$
		so we consider $r_1,r_2\in[m]$ such that $(M\vec x^1)_{r_1} = \max_k (M\vec x^1)_k$ and $(M\vec x^2)_{r_2} = \max_k (M\vec x^2)_k$. We then have 
		\begin{equation}
			\begin{aligned}
				\label{eq:z_diff}
				&z^{\vec x^1}_{r_1}-z^{\vec x^2}_{r_1} \geq 0,\\
				&z^{\vec x^1}_{r_2}-z^{\vec x^2}_{r_2} \leq 0.
			\end{aligned}
		\end{equation}
		Hence,
		\begin{align}
			\dtr^{n}(\vec z^{\vec x^1},\vec z^{\vec x^2})&=\|\vec z^{\vec x^1} - \vec z^{\vec x^2}\|_{\mathrm{tr}}\nonumber\\
			&= \max_{a,b\in[n]} \{ z^{\vec x^1}_a-z^{\vec x^2}_a - z^{\vec x^1}_b + z^{\vec x^2}_b \}\nonumber\\
			&=\max_{a,b\in\cup_iJ_i} \{z^{\vec x^1}_a-z^{\vec x^2}_a - z^{\vec x^1}_b + z^{\vec x^2}_b  \}\nonumber\\
			&= \max_{r,s\in[m]} \{ (M\vec x^1)_{r} - (M\vec x^2)_{r} - (M\vec x^1)_{s} + (M\vec x^2)_{s} \}\nonumber\\
			&= \dtr^{m}(M\vec x^1,M\vec x^2). \label{eq:z:ineq}
		\end{align}                
		Notice that \eqref{eq:z_diff} allows us to restrict the maximum to $\cup_iJ_i$, as for all $j$ not in $\cup_i J_i$ we have that $z^{\vec x^1}_j - z^{\vec x^2}_j = 0$. 
		
		We now prove the first half of \eqref{eq:metric_equivalence}. Firstly, by the triangle inequality, \eqref{eq:z:ineq}, and the non-expansivity of $M$ (\Cref{prop:Non-Expansive_Map}), we note:
		\begin{align*}
			\frac{1}{3}\dtr^{n}(\vec x^1-\vec z^{\vec x^1},\vec x^2-\vec z^{\vec x^2}) &= \frac{1}{3} \|\vec x^1-\vec z^{\vec x^1}-\vec x^2+\vec z^{\vec x^2}\|_{\mathrm{tr}} \\
			&\leq  \frac{1}{3} \left( \|\vec x^1-\vec x^2\|_{\mathrm{tr}} + \|\vec z^{\vec x^2}-\vec z^{\vec x^1} \|_{\mathrm{tr}} \right)\\
			&= \frac{1}{3} \left(\dtr^{n}(\vec x^1,\vec x^2) + \dtr^{m}(M\vec x^1,M\vec x^2)\right) \\
			&\leq \frac{2}{3}\dtr^{n}(\vec x^1,\vec x^2).
		\end{align*}
		Adding this to the non-expansive inequality gives us the left hand inequality of \eqref{eq:metric_equivalence}:
		\begin{align*}
			\frac{1}{3}\dtr^{m}(M\vec x^1,M\vec x^2) + \frac{1}{3}\dtr^{n}(\vec x^1-\vec z^{\vec x^1},\vec x^2-\vec z^{\vec x^2}) \leq \dtr^{n}(\vec x^1,\vec x^2).
		\end{align*}
		We now prove the second half of \eqref{eq:metric_equivalence}:
		\begin{align*}
			\dtr^{n}(\vec x^1,\vec x^2) &= \|\vec x^1-\vec z^{\vec x^1}-\vec x^2+\vec z^{\vec x^2}+\vec z^{\vec x^1}-\vec z^{\vec x^2}\|_{\mathrm{tr}}\\
			&\leq \|\vec x^1-\vec z^{\vec x^1}-\vec x^2+\vec z^{\vec x^2}\|_{\mathrm{tr}}+\|\vec z^{\vec x^1}-\vec z^{\vec x^2}\|_{\mathrm{tr}}\\
			&= \dtr^{n}(\vec x^1-\vec z^{\vec x^1},\,\vec x^2-\vec z^{\vec x^2}) + \dtr^{m}(M\vec x^1,\,M\vec x^2),
		\end{align*}
		where we used the triangle inequality and \eqref{eq:z:ineq}.
		This shows strong equivalence of the metrics, proving that $f$ is a homeomorphism.
	\end{proof}
	
	This product space structure of $\TPT{n}$ will allow us to prove that when optimising over simple projections, the projective and embedding Wasserstein distances on tropical projective tori are the same.
	
	\section{Tropical Wasserstein Mappings for Different Dimensions}
	\label{sec:main_thm}
	
	Given the tropical tools and results established above, we are now equipped to present our main theoretical result; when we use simple projections as our inter-dimensional tropical maps, the projection and embedding Wasserstein distances coincide as in the Euclidean case.
	
	Following the presentation of our main result in this section, we discuss implications of our work in the setting of phylogenetic trees.
	
	\subsection{Equivalent Tropical Wasserstein Distances}
	
	We define our set of simple projections from $\TPT{n}$ to $\TPT{m}$ by 
	\[
	\mathcal{M}_{\mathrm{tr}} \coloneqq \{ \phi_M: \TPT{n} \rightarrow \TPT{m} : \phi_M(x) = Mx \text{ where $M$ is a simple projection} \}.
	\]
	We note each $\phi \in \mathcal{M}_{\mathrm{tr}}$ is measurable as it is continuous.
	
	The sets of projected and embedded measures $\Phi_{\mathrm{tr}}^-$ and $\Phi_{\mathrm{tr}}^+$ are now given by
	\begin{gather*}
		\Phi_{\mathrm{tr}}^-(\nu,m) \coloneqq \{ \beta \in P(\TPT{m}) : \beta = \phi_*(\nu) \text{ for some } \phi \in \mathcal{M}_{\mathrm{tr}}  \}, \label{eq:ProjMeasuresTrop}\\
		\Phi_{\mathrm{tr}}^+(\mu,n) \coloneqq \{ \alpha \in P(\TPT{n}) : \mu = \phi_*(\alpha) \text{ for some } \phi \in \mathcal{M}_{\mathrm{tr}}  \}, \label{eq:EmbdMeasuresTrop}
	\end{gather*}
	while the tropical projection and embedding Wasserstein distances are given by 
	\begin{align*}
		W_{\mathrm{tr},p}^-(\mu,\nu) = \inf_{\beta \in \Phi_{\mathrm{tr}}^-(\nu,m)} W_{\mathrm{tr},p}(\mu, \beta), \label{eq:TropProjWDist} \\
		W_{\mathrm{tr},p}^+(\mu,\nu) = \inf_{\alpha \in \Phi_{\mathrm{tr}}^+(\mu,n)} W_{\mathrm{tr},p}(\alpha, \nu). 
	\end{align*}
	
	We begin by establishing the following lemma, which verifies that simple projections are non-expansive on $P_p(\TPT{n})$ as well as on $\TPT{n}$.
	
	\begin{proposition}
		\label{prop:NonExpansiveOT}
		For all $p \in [0,\infty), \phi \in \mathcal{M}_{\mathrm{tr}}$ and $\alpha, \nu \in P_p(\TPT{n})$, we have
		\[
		W_{\mathrm{tr},p}(\phi_*(\alpha), \phi_*(\nu)) \leq W_{\mathrm{tr},p}(\alpha, \nu).
		\]
	\end{proposition}
	
	\begin{proof}
		This proof follows the same approach as the proof of Lemma 2.1 by \cite{cai2022distances}.
		
		Let $\pi \in P((\TPT{n})^2)$ be the optimal transport coupling for $\alpha$ and $\nu$. Then define $\pi_m \in P((\TPT{m})^2)$ as the push-forward of $\pi$ through $\phi \times \phi$:
		\[
		\pi_m(A \times B) = \pi( \phi^{-1}(A) \times \phi^{-1}(B)).
		\]
		Checking the marginals of $\pi_m$, we have
		\begin{align*}
			\pi_m(A,\TPT{m}) = \pi(\phi^{-1}(A) \times \TPT{n}) = \alpha(\phi^{-1}(A)) = \phi_*(\alpha)(A), \\
			\pi_m(\TPT{m},B) = \pi(\TPT{n} \times \phi^{-1}(B)) = \nu(\phi^{-1}(B)) = \phi_*(\nu)(B).
		\end{align*}
		Therefore, $\pi_m$ is a coupling of $\phi_*(\alpha)$ and $\phi_*(\nu)$. We can then bound the Wasserstein distance between $\phi_*(\alpha)$ and $\phi_*(\nu)$ by:
		\begin{align*}
			W_{\mathrm{tr},p}(\phi_*(\alpha), \phi_*(\nu))^p 
			&\leq \E_{(X,Y) \sim \pi_m}[d_{\mathrm{tr}}(X,Y)^p] \\
			&= \E_{(U,V) \sim \pi}[d_{\mathrm{tr}}(\phi_*(U),\phi_*(V))^p]\\
			&\leq \E_{(U,V) \sim \pi}[d_{\mathrm{tr}}(U,V)^p]\\
			&= W_{\mathrm{tr},p}(\alpha,\nu)^p,
		\end{align*}
		where for the last inequality, we used the non-expansive property (\Cref{prop:Non-Expansive_Map}).
	\end{proof}
	
	\begin{remark}
		Notice that although stated for simple projections, in the proof, we only used the non-expansive property so the proposition holds for any tropical matrix map.
	\end{remark}
	
	As well as relating Wasserstein distances across spaces, \Cref{prop:NonExpansiveOT} gives us the following corollary ensuring finite moments.
	
	\begin{corollary}
		\label{cor:FiniteMoments}
		For any $\nu \in P_p(\TPT{n})$ and any $\phi\in\mathcal{M}_{\mathrm{tr}}$, the push-forward $\phi_*(\nu)$ has finite $p$th moment; that is, $\phi_*(\nu) \in P_p(\TPT{m})$. 
	\end{corollary}
	
	\begin{proof}
		In order to prove that $\phi_*(\nu)$ has finite $p$th moment, we only need to show that there exists a measure with finite $p$th moment from which $\phi_*(\nu)$ has finite $p$-Wasserstein distance. Consider the Dirac measure on $\TPT{m}$, concentrated in $\vec 0$. This can be seen as the push-forward of a Dirac measure in $\TPT{n}$ with the support in any point of $\vec x_0\in F_{\vec 0}$.  Let $\delta_{\vec x}$ be the Dirac measure concentrated at $\vec x$.  Then we can write
		\begin{equation*}
			\begin{aligned}
				W_{\mathrm{tr},p}(\delta_{\vec 0},\phi_*(\nu))=W_{\mathrm{tr},p}(\phi_*(\delta_{\vec x_0}),\phi_*(\nu))\leq W_{\mathrm{tr},p}(\delta_{\vec x_0},\nu) < \infty.
			\end{aligned}
		\end{equation*}
		Here, notice the last term is finite since $\nu$ has finite moments by definition.
	\end{proof}
	
	We now state and prove our main theorem: the equivalence of the tropical projection and embedding Wasserstein distances $W_{\mathrm{tr},p}^-$ and $W_{\mathrm{tr},p}^+$.
	
	\begin{theorem}\label{thm:TropicalMainThm}
		Let $p \in [1,\infty)$. For all $\mu \in P_p(\TPT{m})$ and $\nu \in P_p(\TPT{n})$,
		\[
		W_{\mathrm{tr},p}^-(\mu,\nu) = W_{\mathrm{tr},p}^+(\mu,\nu) < \infty.
		\]
	\end{theorem}
	
	\begin{proof}
		The various results in \Cref{sec:simple_projections} enable us to follow the same approach as in Theorem 2.2 of \cite{cai2022distances}.
		
		We begin by proving $W_{\mathrm{tr},p}^- \leq W_{\mathrm{tr},p}^+$. By \Cref{prop:NonExpansiveOT},
		\begin{align*}
			W_{\mathrm{tr},p}^+(\mu,\nu) &= \inf_{\phi \in \mathcal{M}_{\mathrm{tr}}} \{ W_{\mathrm{tr},p}(\alpha, \nu) : \mu = \phi_*(\alpha)\} \\
			&\geq  \inf_{\phi \in \mathcal{M}_{\mathrm{tr}}} \{ W_{\mathrm{tr},p}(\mu, \phi_*(\nu))\}\\
			&= W_{\mathrm{tr},p}^-(\mu,\nu).
		\end{align*}    
		Note that each $W_{\mathrm{tr},p}(\mu, \phi_*(\nu))$ is finite by \Cref{cor:FiniteMoments}. Hence, $W_{\mathrm{tr},p}^+(\mu,\nu)$ is finite. 
		
		In order to show that $W_{\mathrm{tr},p}^-(\mu,\nu) \leq W_{\mathrm{tr},p}^+(\mu,\nu)$, we prove that for all $\epsilon > 0$, and $\beta_{\epsilon} \in \Phi_{\mathrm{tr},p}^-(\nu, m)$ such that $W_p(\mu, \beta_{\epsilon}) \leq W_p^-(\mu,\nu) + \epsilon$, there exists $\alpha_{\epsilon} \in \Phi_{\mathrm{tr}}^+(\mu,n)$ such that $W_{tr,p}(\alpha_{\epsilon},\nu) \leq W_{\mathrm{tr},p}(\mu,\beta_{\epsilon})$. Indeed, if our claim is true, then we have $\forall\ \epsilon>0$:
		\begin{align*}
			W_{\mathrm{tr},p}^+(\mu,\nu) &\leq W_{\mathrm{tr},p}(\alpha_{\epsilon}, \nu) \\
			&\leq W_{\mathrm{tr},p}(\mu,\beta_{\epsilon}) \\
			&\leq W_p^-(\mu,\nu) + \epsilon.
		\end{align*}
		Letting $\epsilon$ tend to $0$ gives us the desired result.
		
		Now let $\phi_M$ be the simple projection sending $\nu$ to $\beta_{\epsilon}$, i.e., $\beta_{\epsilon}=\phi_M(\nu)$. Using the homeomorphism $f_M$ as defined in \Cref{prop:homeomorphism}, we define the complementary projection to $\phi_M$ as $\phi_M^{F_{\vec 0}}$; that is, $\phi_M^{F_{\vec 0}} = \text{proj}_2 \circ f_M$. We then construct $\alpha_{\epsilon}$ by defining a coupling on $\TPT{n} \times \TPT{n}$ with marginals $\nu,\alpha_{\epsilon}$. We first consider $\TPT{m} \times F_{\vec 0} \times \TPT{m}$ and look to apply the gluing lemma to $f_M(\nu)$ and $\pi_m$. 
		
		\begin{lemma}{\cite[Chapter 1]{villani2009optimal}}
			Let $(\mathcal{X}_i, \mu_i)$, $i=1,2,3$ be Polish spaces. Let $(X_1,X_2)$ be a coupling of $(\mu_1,\mu_2)$ and $(X_2,X_3)$ be a coupling of $(\mu_2,\mu_3)$. Then there exists a coupling of random variables $(Z_1,Z_2,Z_3)$ such that $(Z_1,Z_2) \sim (X_1,X_2)$ and $(Z_2,Z_3) \sim (X_2,X_3)$.
		\end{lemma}
		
		The space $\TPT{m}$ is Polish, while $F_{\vec 0}$ is a closed subset of a Polish space and is therefore also Polish. Hence, by the gluing lemma, there exist random vectors $(X,Y,Z)$ on $\TPT{m} \times F_{\vec 0} \times \TPT{m}$ such that $(X,Y) \sim f_M(\nu)$ and $(X,Z) \sim \pi_m$. We call the distribution of $(X,Y,Z) \sim \tilde \pi$.
		
		We now define $\rho: \TPT{m} \times F_{\vec 0} \times \TPT{m} \rightarrow (\TPT{n})^2$ given by
		\[
		\rho(\vec x,\vec y,\vec z) \rightarrow (f_M^{-1}(\vec x,\vec y),\, f_M^{-1}(\vec z, \vec y)).
		\]
		This is a measurable map as $f_M^{-1}$ is continuous.
		
		We then define a coupling $\pi_n$ on $(\TPT{n})^2$ and measure $\alpha_{\epsilon}$ by
		$\pi_n = \rho(\tilde \pi)$ and $
		\alpha_{\epsilon} = \text{proj}_2(\pi_n)$.
		
		It remains to show that $\pi_n$ is a coupling of $\nu,\alpha_{\epsilon}$, the distance $W_{\mathrm{tr},p}(\alpha_{\epsilon},\nu)$ is bounded by $W_{\mathrm{tr},p}(\mu, \beta_{\epsilon})$, and that $\alpha_{\epsilon} \in \Phi^+(\mu,n)$.
		
		Since $\alpha_{\epsilon}$ is the second marginal of $\pi_n$ by definition, it suffices to compute the first marginal:
			$$
			\pi_n(A, \TPT{n}) = \tilde\pi(\rho^{-1}(A,\TPT{n}))= \tilde\pi(f_M(A),\TPT{m}) = \nu(A).
			$$
		We now want to bound $W_{\mathrm{tr},p}(\alpha_{\epsilon},\nu)$ using $\pi_n$. We begin by showing that for any $(\vec x,\vec y)\in \TPT{m} \times F_{\vec 0}$ and $\vec z\in\TPT{m}$, we have  
		$$
		\dtr^{n}(\rho(\vec x,\vec y,\vec z)) = \dtr^{m}(\vec x, \vec z).
		$$
		We obtain the lower bound on $\dtr^{n}(\rho(\vec x,\vec y,\vec z))$ through the following argument:
		\begin{align*}
			\dtr^{n}(\rho(\vec x,\vec y,\vec z)) &=\dtr^{n}(f_M^{-1}(\vec x,\vec y),\, f_M^{-1}(\vec z, \vec y)) \\
			&\geq \dtr^{m}(\phi_M(f_M^{-1}(\vec x, \vec y)),\, \phi_M(f_M^{-1}(\vec z, \vec y))) \\
			&= \dtr^{m}(\vec x, \vec z),
		\end{align*}
		in which we used \Cref{prop:Non-Expansive_Map} for the inequality, while the equalities come from the definitions of $\rho$ and $f_M$, respectively.
		Using the upper bound of \eqref{eq:metric_equivalence} gives the upper bound:
		\begin{align*}
			\dtr^{n}(\rho(\vec x,\vec y,\vec z)) &=\dtr^{n}(f_M^{-1}(\vec x,\vec y),\, f_M^{-1}(\vec z, \vec y)) \\
			&\leq \dtr^{m}(\vec x, \vec z) + \dtr^{n}(\phi_M^{F_\vec 0}(f_M^{-1}(\vec x,\vec y)),\, \phi_M^{F_\vec 0}(f_M^{-1}(\vec z,\vec y))) \\
			&= \dtr^{m}(\vec x, \vec z) + \dtr^{n}(\vec y, \vec y)= \dtr^{m}(\vec x, \vec z).
		\end{align*}
		This gives us the following upper bound for $W_{\mathrm{tr},p}(\alpha_{\epsilon},\nu)$:
		\begin{align*}
			W_{\mathrm{tr},p}(\alpha_{\epsilon},\nu)^p &\leq \E_{U,V\sim \pi_n}[\dtr^{n}(U,V)^p] \\
			&= \E_{X,Y,Z \sim \tilde\pi}[\dtr^{n}(\rho(X,Y,Z))^p]\\
			&= \E_{X,Y,Z \sim \tilde\pi}[\dtr^{m}(X,Z)^p] \\
			&= \E_{X,Z \sim \pi_m}[\dtr^{m}(X,Z)^p]\\
			&= W_{\mathrm{tr},p}(\mu,\beta_{\epsilon})^p.
		\end{align*}
		Finally, it only remains to show that $\alpha_{\epsilon} \in \Phi^+(\mu,n)$, i.e., $\phi_M(\alpha_{\epsilon}) = \mu$:
		\begin{align*}
			\phi_M(\alpha_{\epsilon})(A) &= \pi_n(\TPT{n},\, \phi_M^{-1}(A))\\
			&= \tilde\pi(\rho^{-1}(\TPT{n},\, \phi_M^{-1}(A)))\\
			&= \tilde\pi(\TPT{m}, F_{\vec 0}, A) \\
			&= \pi_m(\TPT{m},A) \\
			&= \mu(A).
		\end{align*}
		
	\end{proof}
	
	Given this equivalence, from now on, we now denote $W_{\mathrm{tr},p}^-,W_{\mathrm{tr},p}^+$ by $W_{\mathrm{tr},p}^{m,n}$.
	
	\begin{remark}\label{rmk:trop_remarks}
		We stress that this proof only relies on $\phi \in \mathcal{M}$ being a simple projection, and not $\mathcal{M}_{\mathrm{tr}}$ exhibiting any closure. We can therefore define $\Phi_{\mathrm{tr}}^{\pm}, W_{\mathrm{tr},p}^{m,n}$ with respect to any subset of maps $\mathcal{M}_{\mathrm{tr}}' \subset \mathcal{M}_{\mathrm{tr}}$ and reach the same result.
		
		However, it is important to note that the quantity $W_{\mathrm{tr},p}^{m,n}(\mu,\nu)$ is not a formal metric. Indeed, we can have $W_{\mathrm{tr},p}^{m,n}(\mu,\nu)=0$ for $\mu\neq\nu$ if, say, $\phi\in\mathcal{M}_{\mathrm{tr}}$, such that $\nu=\phi_*(\mu)$. 
		
	\end{remark}
	
	\begin{proposition} \label{prop:consistency}
		The distance $W_{\mathrm{tr},p}^{m,n}$ is 1-Lipschitz continuous on $(P_p(\TPT{m}), W_{\mathrm{tr},p}) \times (P_p(\TPT{n}), W_{\mathrm{tr},p})$. 
	\end{proposition}
	
	\begin{proof}
		\begin{align*}
			|W_{\mathrm{tr},p}^{m,n}(\mu_1,\nu_1) -& W_{\mathrm{tr},p}^{m,n}(\mu_2,\nu_2)| = \left|\inf_{\phi \in \mathcal{M}_{\mathrm{tr}}} W_{\mathrm{tr},p}(\mu_1,\phi_*(\nu_1)) - \inf_{\phi \in \mathcal{M}_{\mathrm{tr}}} W_{\mathrm{tr},p}(\mu_2,\phi_*(\nu_2))\right| \\
			\leq&|\inf_{\phi \in \mathcal{M}_{\mathrm{tr}}} \left[ W_{\mathrm{tr},p}(\mu_1,\mu_2) + W_{\mathrm{tr},p}(\mu_2,\phi_*(\nu_2)) + W_{\mathrm{tr},p}(\phi_*(\nu_2),\phi_*(\nu_1)) \right] \\
			&- \inf_{\phi \in \mathcal{M}_{\mathrm{tr}}} W_{\mathrm{tr},p}(\mu_2,\phi_*(\nu_2)) | \\ 
			\leq& |\inf_{\phi \in \mathcal{M}_{\mathrm{tr}}} \left[W_{\mathrm{tr},p}(\mu_1,\mu_2) + W_{\mathrm{tr},p}(\mu_2,\phi_*(\nu_2)) + W_{\mathrm{tr},p}(\phi_*(\nu_2),\phi_*(\nu_1)) \right] \\
			&- \inf_{\phi \in \mathcal{M}_{\mathrm{tr}}} W_{\mathrm{t}r,p}(\mu_2,\phi_*(\nu_2))| \\ 
			\leq& |W_{\mathrm{tr},p}(\mu_1,\mu_2) + W_{\mathrm{tr},p}(\phi_*(\nu_2),\phi_*(\nu_1))\\
			&+ \inf_{\phi \in \mathcal{M}_{\mathrm{tr}}} \left[W_{\mathrm{tr},p}(\mu_2,\phi_*(\nu_2)) \right] - \inf_{\phi \in \mathcal{M}_{\mathrm{tr}}} W_{\mathrm{tr},p}(\mu_2,\phi_*(\nu_2))| \\
			&\leq |W_{\mathrm{tr},p}(\mu_1,\mu_2) + W_{\mathrm{tr},p}(\nu_2,\nu_1)| \\
			&\leq W_{\mathrm{tr},p}(\mu_1,\mu_2) + W_{\mathrm{tr},p}(\nu_2,\nu_1).
		\end{align*}
	\end{proof}
	
	This in turn gives the following implication for estimation of $W_{\mathrm{tr},p}^{m,n}(\mu,\nu)$.
	
	\begin{corollary}
		Let $X=\{\vec x_1,\dots,\vec x_r\}\subset\TPT{n}$ and $Y=\{\vec y_1,\dots,\vec y_s\}\subset\TPT{m}$ sampled independently from $\mu \in P_p(\TPT{m})$ and $\nu \in P_p(\TPT{n})$, respectively. Consider $\hat \mu_r$ and $\hat \nu_s$, the empirical measures of $X$ and $Y$.Then $\widehat W_{\mathrm{tr},p}^{m,n} = W_{\mathrm{tr},p}^{m,n}(\hat \mu_r, \hat \nu_s)$ is a consistent estimator for $W_{\mathrm{tr},p}^{m,n}(\mu,\nu)$.
	\end{corollary}
	
	\begin{proof}
		By \Cref{prop:consistency} and the fact that Wasserstein distances respect weak convergence of measures \citep[e.g.,][Theorem 6.9]{villani2009optimal}, we only require the weak convergence of $\hat \mu_r$(/$\hat \nu_s$) to $\mu$(/$\nu$) as $r$(/$s$) tend to infinity. This is true by the law of large numbers as $\TPT{n}$ is a Polish space \citep{monod2018StatPerspective}.
	\end{proof}
	
	In this section we have defined simple projection matrices, the most general tropically linear maps through which we can project and embed measures while preserving non-expansivity and co-dimension. We have shown that the projection and embedding Wasserstein distances (defined with respect to simple projections) coincide as in the Euclidean case, despite the more complex geometry of the tropical projective torus. Furthermore, we have shown that this Wasserstein distance is Lipschitz, and so can be estimated consistently using samples from measures $\mu, \nu$.
	
	\subsection{Open Problems and Implications: Phylogenetic Trees}\label{subsec:open_problems}
	
	To conclude this section, we return to the application of interest which motivated our work --- tropical phylogenetic tree space.  We have constructed projection and embedding Wasserstein distances between probability measures in tropical projective tori of different dimensions, and proved that the projection and embedding distances are equivalent. We would like to be able to use this tool to compare probability distributions over phylogenetic tree spaces --- which are subspaces within tropical projective tori --- where the numbers of leaves in the trees differ.  There are, however, important theoretical considerations to take into account which are specific to phylogenetic trees.
	
	The tropical space of phylogenetic trees with $N$ leaves is characterised by those points in $\TPT{N(N-1)/2}$ satisfying the four-point condition \citep{buneman1974note}.
	
	\begin{definition}[Four-Point Condition]
		Let $d$ be a metric on $N$ points. Then $d$ corresponds to the graph distance on some tree with $N$ leaves if and only if, $\forall i,j,k,l \in [N]$, the maximum of the following three terms is achieved twice:
		\[
		d(i,j)+d(k,l), \qquad d(i,k)+d(j,l), \qquad d(i,l)+d(j,k).
		\]
	\end{definition}
	
	\begin{definition}[Tropical Tree Space]
		The tropical tree space on $N$ leaves, denoted by $\mathcal{T}_N$, is the set of vectors $\vec x \in \TPT{N(N-1)/2}$ such that $\{x_{ij}\}_{i\neq j \in [N]}$ satisfies the four-point condition. 
	\end{definition}
	
	In light of these definitions, when calculating Wasserstein distances between measures which are supported on tropical tree space, it is natural to ask; which simple projections map tropical tree space into tropical tree space? Do such maps necessarily minimise the Wasserstein distance? The advantages of studying these simple projections are two-fold; we may be able to speed up computations by computing an infimum over a subset of simple projections rather than over all simple projections, and a theoretical understanding of such maps may allow us to interpret the optimal simple projection as interpretable relations between the two leaf sets. Formally, we pose these questions in the following way:
	
	\begin{open-prob}\label{op:parameterisation}
		Let $\mathcal{MT_{\tr}}$ be the set of simple projections $M$ such that, for all $\vec x \in \TPT{N(N-1)/2}$ satisfying the four-point condition, $M\vec x \in \TPT{M(M-1)/2}$ satisfies the four-point condition. Can we find a parametric description of the set $\mathcal{MT_{\tr}}$?
	\end{open-prob}
	\begin{open-prob}\label{op:metric-minimisation}
		Suppose $\mu \in P_p(\mathcal{T_N})$ and $\nu \in P_p(\mathcal{T_M})$. Is it true that
		\[
		W_{\mathrm{tr},p}^{m,n}(\mu, \nu) \coloneqq \inf_{M \in \mathcal{M_{\tr}}} W_{\tr,p}(M\mu, \nu) = \inf_{M \in \mathcal{MT_{\tr}}} W_{\tr,p}(M\mu, \nu)?
		\]
	\end{open-prob}
	
	A natural approach to \Cref{op:parameterisation} may involve utilising the deep connections between tropical geometry and matroid theory \citep{ardila2006bergman}, which is beyond the scope of this paper. An affirmative proof of \Cref{op:metric-minimisation} would most likely require the parametrisation of $\mathcal{MT}_{\tr}$ from \Cref{op:parameterisation}, but it may be possible to directly construct a counter-example by exploiting the lack of tropical convexity of tropical tree space.
	
	It may also be useful to first address simpler versions of \Cref{op:parameterisation,op:metric-minimisation} by restricting to the case of \emph{ultrametric trees}. These are special cases of phylogenetic trees which are \emph{equidistant}; they are rooted phylogenetic trees where the distance from every leaf to its root is a constant. The set of such trees forms a tropical linear space, and hence is tropically convex \citep{lin2017convexity}. This convexity has been recently used to characterise the combinatorial behaviour of tree topologies in the tropical setting \citep{lin2022tropical}. The \Cref{op:parameterisation,op:metric-minimisation} can also be posed with respect to ultrametric tree space rather than general tropical tree space, and may prove to be more tractable by utilising the tropical convexity of ultrametric space.
	
	\section{Computation}\label{sec:computation}
	
	Having defined tropical projection and embedding Wasserstein distances with sufficiently nice theoretical properties, in this section we demonstrate how existing optimisation techniques can be used for the computation of $W_{\mathrm{tr},p}^{m,n}$. The computation of $W_{\mathrm{tr},p}^{m,n}$ is not a simple optimisation problem, as we require a minimal Wasserstein distance over the set of simple projections, which is a disconnected set. \Cref{lemma:optimser} shows that we can separate this problem into discrete and continuous parts by expressing $W_{\mathrm{tr},p}^{m,n}$ as a minimisation problem over three components; the support of real entries in the simple projection matrix, the exact real entries of our matrix, and the joint weights $\lambda_{r,s}$ between the data points $\vec x^r$ and $\vec y^s$. The proof of this lemma can be found in \Cref{appsec:proofs}.
	
	\begin{lemma}\label{lemma:optimser}
		Let $X = \{\vec x\}_{r \in [R]}$ and $Y = \{\vec y\}_{s \in [S]}$ be finite datasets in $\TPT{n}$ and $\TPT{m}$ respectively. The Wasserstein distance $W_{\mathrm{tr},p}^{m,n}(X,Y)$ is the infimum of the function
		\begin{equation}\label{eq:objective}
			f(M,\lambda, \vec t) =  \left(\sum_{r,s} \lambda_{r,s} \| M( \mathbf{x^{(r)}} + \mathbf{t}) - \mathbf{y}\|_{tr}^p \right)^{1/p}
		\end{equation}
		
		where $\vec t \in \TPT{N}$, $\lambda$ is a joint probability matrix with uniform marginals on $[R]$ and $[S]$, and $M \in \mathcal{M_{\tr}'}$, where $\mathcal{M_{\tr}'}$ is the set of simple projections with all real entries equal to zero and exactly one real entry per column.
	\end{lemma}
	
	In this section we first outline the existing optimisation methods which can be used to optimise over $M, \lambda$, and $\vec t$ independently, before combining these methods to present our coordinate descent algorithm (\Cref{algo:minimisation}) for the computation of $W_{\mathrm{tr},p}^{m,n}(X,Y)$. We then apply this algorithm to simulated branching process data, coalescent data, and Gaussian data. To conclude, we compute the Wasserstein distances between phylogenetic influenza datasets across different annual seasons. The full results of numerical experiments are presented in \Cref{appsec:gradient_methods,appsec:simulated_annealing}; the implementation for these experiments and the computations on simulated and real data are available in the GitHub repository at \url{https://github.com/Roroast/WDDD}. 
	
	\subsection{Algorithm}
	
	We take a coordinate descent approach to the minimisation of the objective given in \Cref{eq:objective}, iteratively minimising a single component while fixing the others. To optimise the shift component $\vec t$, we will consider gradient methods as $f$ is piece-wise polynomial in terms of $\vec t$. We will implement simulated annealing to solve the discrete optimisation problem of the optimal matrix support $M$, while optimising over $\lambda$ is a classic discrete optimal transport problem, which is a linear program with efficient solvers \citep{flamary2021pot}. In the following subsections we review various options for the gradient based optimisation of $\vec t$ and possible graph structures for simulated annealing in $M$, with full numerical experiments to compare the performances of various methods in \Cref{appsec:gradient_methods,appsec:simulated_annealing}.
	
	\subsubsection{Gradient Methods and Learning Rate}\label{subsubsec:shift_search} The objective function given in \Cref{eq:objective} --- for fixed $M, \lambda$ --- is piecewise polynomial in $\vec t$ with regions of sparse gradient. This raises what is known as the zig-zagging problem; when outside the tropical convex hull of the data, the sparse gradient of $f$ can cause gradient methods to zig-zag across a ``valley'' of local minima. This behaviour is explored by \cite{talbut2024descent}, where it is shown that tropical gradient methods can achieve smaller error on average and avoid convergence to local (non-global) minima.
	
	In this paper we consider classical gradient descent (CD), tropical gradient descent (TD), Adamax (CA) and TrAdamax (TA) for the optimisation of $\vec t$, and provide an outline of these methods below. For a detailed explanation and motivation for their implementation, see \cite{talbut2024descent}. 
	
	\emph{Classical gradient descent(CD)} follows the steepest descent algorithm with respect to the Euclidean norm. This is given by the following iterative rule:
	\begin{equation}\label{eq:C_descent_direction}
		\vec t_{k+1} = \vec t_{k} - \alpha_{k} \vec d_{k}, \qquad \alpha_k = \alpha k^{-1/2} \| \nabla f(\vec t_k) \|, \qquad \vec d_k =\nabla f(\vec t_k).
	\end{equation}
	
	\emph{Tropical gradient descent(TD)} follows the steepest descent algorithm with respect to the tropical norm, which leads to the following iterative rule:
	\begin{equation}\label{eq:T_descent_direction}
		\vec t_{k+1} = \vec t_{k} - \alpha_{k} \vec d_{k}, \qquad \alpha_k = \alpha k^{-1/2} \| \nabla f(\vec t_k) \|_{\tr}, \qquad \vec d_k =\mathbf{1}_{\nabla f(\vec t_k) > 0}.
	\end{equation}
	
	\emph{Adamax(CA)} is a variant of the Adam algorithm \citep{kingma2014adam}, which uses exponentially weighted moments of the gradient to act as momentum terms. These momentum terms can independently adapt step sizes in each coordinate direction appropriately. While Adam uses $L^2$ moments of the gradient, Adamax uses $L^{\infty}$ moments and is shown to behave better than Adam in tropical optimisation problems \citep{talbut2024descent}.
	
	\emph{TrAdamax(TA)} is a heuristic adaptation of the Adamax algorithm, which uses exponentially weighted moments of the tropical descent directions $\mathbf{1}_{\nabla f(\vec t_k) > 0}$, rather than the gradient $\nabla f(\vec t_k)$.
	
	\begin{table}[h!]
		\centering
		\begin{tabular}{|c|c|c|r|r|r|r|r|r|r|r|}
			\hline
			& \multicolumn{10}{|c|}{Learning rate, $\alpha$} \\
			\hline
			Grad &  $e^{-5}$  &  $e^{-4}$  &  $e^{-3}$  &  $e^{-2}$  &  $e^{-1}$  &  $e^{0}$  &  $e^{1}$  &  $e^{2}$  &  $e^{3}$ & $e^{4}$ \\
			\hline
			CA   &      2.047 &      1.086 &      1.074 &      \textbf{1.070} &      1.073 &      1.092 &      1.201 &      1.427 &      1.932 &      3.192 \\
			CD   &      2.657 &      2.179 &      1.485 &      1.101 &      1.083 &      1.078 &      1.075 &      \textbf{1.073} &      1.123 &      1.803 \\
			TA   &      2.514 &      1.759 &      1.072 &      1.067 &      \textbf{1.066} &      1.070 &      1.110 &      1.220 &      1.693 &      3.085 \\
			TD   &      2.590 &      1.948 &      1.180 &      1.078 &      1.074 &      1.071 &      \textbf{1.066} &      1.068 &      1.174 &      1.719 \\
			\hline
		\end{tabular}
		\caption{The mean loss $f$ over 10 initialisations, 9 sample distributions and 6 dimension pairs $m \leq n \in \{ 6, 10, 15 \}$ achieved by classical Adam, classical descent, TrAdamax, and tropical descent. Learning rate is varied on a linear log scale from $e^{-5}, e^{4}$.}
		\label{tab:lr_tuning_averages}
	\end{table}
	
	Using data of varying dimensionality and distribution, we implemented CA, CD, TA and TD for the minimisation of the loss function $f$ with respect to $\vec t$. We present the full methodology and results in \Cref{appsec:gradient_methods}, while \Cref{tab:lr_tuning_averages} presents the mean loss over all data types. We see that TA (with learning rate $\alpha = e^{-1}$) achieves the best mean loss. The optimal TA learning rates for each data dimension pair is presented in \Cref{tab:trained_lr}.
	
	\begin{table}[h!]
		\centering
		\begin{tabular}{|l|rrrrrr|}
			\hline
			n  & 6 & 10 & 10& 15 & 15 & 15 \\
			\hline
			m  & 6 &  6 & 10 &  6& 10 & 15 \\
			\hline
			Learning rate, $\alpha$ & 0.050 &  1.000 &  0.135 &  0.368 &  0.368 &  0.135 \\
			\hline
		\end{tabular}
		\caption{The TA learning rate $\alpha$ achieving the minimal mean loss $f$ over 10 initialisations and 9 sample distributions for each dimension pairs $m \leq n \in \{ 6, 10, 15 \}$.}
		\label{tab:trained_lr}
	\end{table}
	
	\subsubsection{Simulated Annealing Graph and Scale}\label{subsubsec:supp_search}
	
	The most intensive challenge in minimising the objective \Cref{eq:objective} is the combinatorial problem of identifying the optimal support of our matrix map. As shown in \Cref{lemma:optimser}, it suffices to consider matrices with exactly one real entry in each column as we are computing on finite datasets. We can identify this set of supports with the set of surjections $[n] \rightarrow [m]$, so the size of this set is given explicitly by $m!S(n,m)$, where  $S(n,m)$ is the Stirling number of the second kind. \Cref{tab:support count} shows these values for $m \leq n \leq 6$.
	
	\begin{table}[ht!]
		\centering
		\begin{tabular}{c|c|c|c|c|c|c}
			\backslashbox{$n$}{$m$}& 1 & 2 & 3 & 4 & 5 & 6 \\
			\hline
			1 & 1 &  &  &  &  & \\
			2 & 1 & 2 &  &  &  & \\
			3 & 1 & 6 & 6 &  &  & \\
			4 & 1 & 14 & 36 & 24 &  & \\
			5 & 1 & 30 & 150 & 240 & 120 & \\
			6 & 1 & 62 & 540 & 1560 & 1800 & 720 \\
		\end{tabular}
		\caption{The number of possible matrix supports of $n \times m$ simple projections.}
		\label{tab:support count}
	\end{table}
	
	Due to the exponential growth of the number of supports, we implement simulated annealing \citep{kirkpatrick1983optimization} rather than an exhaustive search to find the optimal support. Simulated annealing is a Metropolis procedure on the set of supports, where we perform a random walk and accept any step which decreases $f$, while accepting increasing steps with probability $\exp( - \Delta f / T(k))$, where $T$ is linearly decreasing in the step number $k$.
	
	To implement simulated annealing, we require a graph structure on our set of possible supports $G = (\mathcal{M_{\tr}'},E)$; we shall consider $E$ to contain exactly those edges between two supports $M_1, M_2$ whose matrices differ in exactly one column or are the same up to column permutations. For any dimensions $n,m$, this results in a connected graph of supports. As a benchmark, we compare the computational performance of this graph $G$ with a complete graph $K$. We also optimise over the scale factor $\sigma$ of the linearly decreasing $T(k)$. The full methodology and results are in \Cref{appsec:simulated_annealing}, while \Cref{tab:scale_tuning_averages} summarises the mean loss of simulated annealing with respect to the graph structure. The incomplete graph $G$ outperforms the complete graph $K$, achieving minimal mean loss with a scale of $e^{-2}$.
	
	\begin{table}[ht]
		\centering
		\begin{tabular}{|l|l|l|r|r|r|r|r|r|r|r|}
			\hline
			& \multicolumn{10}{|c|}{Scale, $\sigma$} \\
			\hline
			Graph &$e^{-7}$  & $e^{-6}$  &  $e^{-5}$  &  $e^{-4}$  &  $e^{-3}$  &  $e^{-2}$  &  $e^{-1}$  &  $e^{0}$  &  $e^{1}$  &  $e^{2}$ \\
			\hline
			K   &     1.066 &     1.065 &     1.064 &     1.064 &     1.066 &     1.062 &     \textbf{1.052} &     1.148 &     1.560 &     1.606 \\
			G	&     1.053 &     1.053 &     1.053 &     1.052 &     1.044 &     \textbf{1.033} &     1.101 &     1.329 &     1.495 &     1.587 \\
			\hline
		\end{tabular}
		\caption{The mean loss over 10 initialisations, 9 sample distributions and 6 dimension pairs $m \leq n \in \{ 6, 10, 15 \}$ achieved by simulated annealing on the $G$ and $K$ graph structures. Scale is varied on a linear log scale from $e^{-7}, e^{2}$.}
		\label{tab:scale_tuning_averages}
	\end{table}
	
	\begin{table}[h!]
		\centering
		\begin{tabular}{|l|rrrrrr|}
			\hline
			n  & 6 & 10 & 10& 15 & 15 & 15 \\
			\hline
			m  & 6 &  6 & 10 &  6& 10 & 15 \\
			\hline
			Scale, $\sigma$ & 0.368 &  0.135 &  0.135 &  0.135 &  0.135 &  0.135 \\
			\hline
		\end{tabular}
		\caption{The simulated annealing scale $\sigma$ achieving the minimal mean loss $f$ over 10 initialisations and 9 sample distributions for each dimension pairs $m \leq n \in \{ 6, 10, 15 \}$.}
		\label{tab:trained_scale}
	\end{table}
	
	\subsubsection{Coordinate Descent Algorithm}
	
	Informed by the numerical experiments presented in \Cref{appsec:gradient_methods,appsec:simulated_annealing}, our complete algorithm for the minimisation of the objective \Cref{eq:objective} is given by \Cref{algo:minimisation}, with the default hyper-parameters given by the optimal learning rate $\alpha$ and scale $\sigma$ shown in \Cref{tab:trained_lr,tab:trained_scale}. To solve the optimal transport linear program for the optimal $\lambda$, we use the POT Python package \citep{flamary2021pot}.
	
	\begin{algorithm}
		\caption{Minimisation}\label{algo:minimisation}
		\begin{algorithmic}
			\Require $\sigma, \alpha, \epsilon > 0, \beta_1, \beta_2 \in [0,1], i_{\max}, j_{\max}, k_{\max} \in \N$
			\State Initialise $\vec t_{00}, M_{00}, \lambda_{0}$
			\State $k \gets 0$
			\While{$k < k_{\max}$}
			\State $k \gets k+1$
			\State $i \gets 1$
			\State $j \gets 1$
			\While{$i < i_{\max}$} \Comment{Perform TrAdamax}
			\State $\vec g_{ki} \gets \nabla f(\vec t_{k(i-1)})$;
			\State $\vec d_{ki} \gets \mathbf{1}_{\vec g_{ki} > 0} * (\max(\vec g_{ki}) - \min(\vec g_{ki}))$;
			\State $\vec m_{ki} \gets \beta_1 * \vec m_{k(i-1)} + (1-\beta_1) * \vec d_{ki}$;
			\State $\vec u_{ki} \gets \max(\beta_2 * \vec u_{k(i-1)}, |\vec d_{ki}|)$;
			\State $\vec t_{ki} \gets \vec t_{k(i-1)} - \alpha/\sqrt{ki}(1-\beta_1^i) * \vec m_{ki}/\vec u_{ki}$;
			\State $i \gets i+1$
			\EndWhile
			\State $\vec t_{(k+1)0} \gets \vec t_{ki}$
			\While{$j < j_{\max}$} \Comment{Perform simulated annealing}
			\State $M \gets$ random neighbour of $M_{k(j-1)}$ in $G$
			\State $\Delta = f(M,\lambda_{k-1},\vec t_{ki}) - f(M_{k(j-1)}, \lambda_{k-1}, \vec t_{ki})$
			\State $T \gets 1 - j/j_{max}$
			\If{$\exp(\sqrt{k}\Delta/\sigma T)< U \sim U[0,1]$}
			\State $M_{kj} \gets M$
			\Else
			\State $M_{kj} \gets M_{k(j-1)}$
			\EndIf
			\State $j \gets j+1$
			\EndWhile
			\State $M_{(k+1)0} \gets M_{kj}$
			\State $\lambda_k \gets \argmin_{\lambda} \, f(M_{kj}, \cdot, \vec t_{ki})$\Comment{Solve optimal transport linear program}
			\If{$|f(M_{(k-1)j}, \lambda_{k-1},\vec t_{(k-1)i}) - f(M_{kj}, \lambda_k,\vec t_{ki})| < \epsilon$}
			\State \Return $f(M_{kj}, \lambda_k,\vec t_{ki})$
			\EndIf
			\EndWhile
			\Return $f(M_k, \lambda_k,\vec t_k)$
		\end{algorithmic}
	\end{algorithm}
	
	\subsection{Simulated Data}
	
	To visualise the results of \Cref{algo:minimisation} and compare distances between distributions of different types, we consider simulated datasets of sample size 50 taken from branching tree and coalescent tree distributions on $\mathcal{T}_4 \subset \TPT{6}, \mathcal{T}_5 \subset \TPT{10}, \mathcal{T}_6 \subset \TPT{15}$, as well as standard Gaussian distributions on $\R^{6}, \R^{10}, \R^{15}$. These distributions are all coordinate invariant, and so symmetric about the origin of the tropical projective torus. We therefore normalise these datasets by the average tropical norm of their data points to remove effects of the dimensionality on the scale of the data.
	
	We then run \Cref{algo:minimisation} for $k_{\max}=40$ steps, and run simulated annealing and TrAdamax for $i_{\max}, j_{\max} = 50$ sub-steps for each $k$ to compute the Wasserstein distances between branching process, coalescent, and Gaussian data in $\TPT{6}, \TPT{10}, \TPT{15}$. For each dataset pair, we run \Cref{algo:minimisation} from 10 random initialisations and take the minimal value over all initialisations. \Cref{fig:sim_data_heatmap} shows the Wasserstein distances between these datasets as a matrix heatmap. We note from \Cref{fig:sim_data_heatmap} that the datasets from coalescent processes have a greater Wasserstein distance to other datasets, even to other coalescent datasets.
	
	\begin{figure}[ht!]
		\centering
		\includegraphics[width = .7\textwidth]{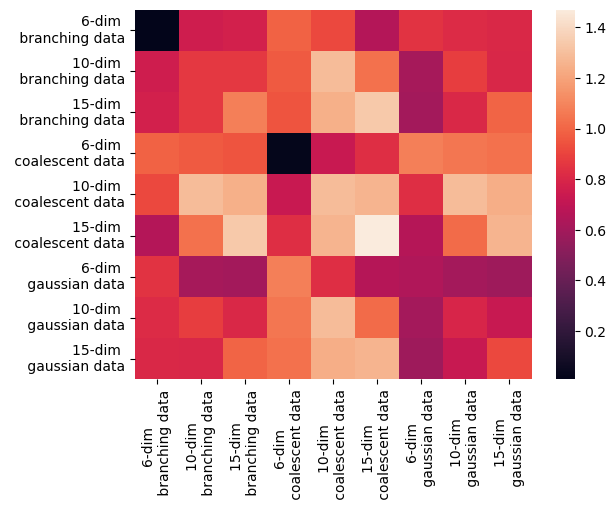}
		\caption{The computed Wasserstein distance between branching, coalescent and Gaussian data of dimension $6, 10$ and $15$. We take a minimum over 10 initialisations of \Cref{algo:minimisation}, and use the tuned learning rate and scale from \Cref{tab:lr_tuning_averages,tab:scale_tuning_averages}.}
		\label{fig:sim_data_heatmap}
	\end{figure}
	
	\subsection{Computation for Influenza Data}
	
	To conclude the computational section of this paper, we use \Cref{algo:minimisation} to compute Wasserstein distances between real phylogenetic datasets which are supported on different datasets.
	
	We use the pre-processed phylogenetic influenza datasets presented by \cite{monod2018StatPerspective}; this phylogenetic data is computed from the genetic sequences of the influenza A H3N2 hemagglutinin obtained for New York, 1993-2017, from the GISAID EpiFlu\texttrademark database \citep{elbe2017data}. The raw genomic data was then aligned using MUSCLE \citep{edgar2004muscle} with default settings, to construct a single large phylogenetic tree for the influenza strains across all seasons. Finally, tree dimensionality reduction \citep{zairis2016genomic} is applied to represent this large phylogenetic tree as a dataset of smaller trees using temporal windows of 5 years; this produces 21 datasets of approximately 20,000 trees each, where the leaf set of a given dataset corresponds to 5 consecutive years. For further details on data pre-processing, see \cite{monod2018StatPerspective}.
	
	From each dataset, we sub-sample 50 tree data points to reduce the computational load on the optimal transport calculations. We use the R package APE~\citep{paradis2004ape} to compute the pairwise leaf distances from the Newick tree format for each tree, to produce 21 datasets of 50 data points of dimension 10. We note that though the dimensionality of each dataset is the same, the dimensions for each dataset correspond to different pairs of years; there is therefore no natural correspondence between the dimensions of the 1993-1997 dataset and the 2013-2017 dataset.
	
	\Cref{fig:real_data_heatmap} shows the computed Wasserstein distances between the 21 datasets. From \Cref{fig:real_data_heatmap}, we observe that all datasets for temporal windows containing 2002 have relatively small distances to each other, while having larger distances to all other seasons. We note that this time window has significant overlap with the prominence of the A/Fujian/411/2002 virus strain, which led to an ineffective trivalent inactivated influenza vaccine in 2003 \citep{centers2004preliminary}. We also observe a small computed Wasserstein distance between the datasets of 1996-2000 and 2003-2007, which may reflect the mutation reassortment events which are discussed by \cite{holmes2005whole}.
	
	\begin{figure}[ht!]
		\centering
		\includegraphics[width = .7\textwidth]{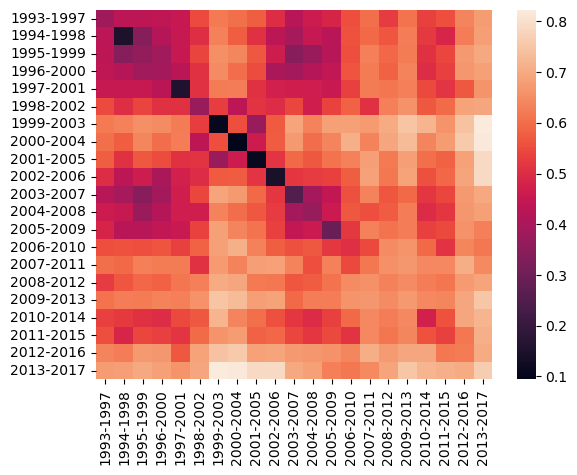}
		\caption{The computed Wasserstein distance between branching, coalescent and Gaussian data of dimension $6, 10$ and $15$. We take a minimum over 10 initialisations of \Cref{algo:minimisation}, and use the tuned learning rate and scale from \Cref{tab:lr_tuning_averages,tab:scale_tuning_averages}.}
		\label{fig:real_data_heatmap}
	\end{figure}
	
	\section{Discussion}
	\label{sec:discussion}
	
	In this manuscript, we have studied the problem of comparing probability distributions over tropical projective tori with different dimensions and formulated an inter-dimensional Wasserstein distance for this purpose.  Our construction was largely inspired by a recent solution to the same problem in Euclidean spaces.  We highlighted the key components of semi-orthogonal maps used in the Euclidean setting and identified tropical matrices with the same properties to be able to arrive at the same conclusion on tropical projective tori. Specifically, we studied tropical simple projections; only the basic properties of non-expansivity, surjectivity, and fibres of consistent dimension were required in order to achieve the same result as in the Euclidean setting, even though the tropical projective torus can be seen as comparatively more restrictive (e.g., it is not a Hilbert space under the tropical norm \citep{amendola2021invitation}).
	
	Considering our motivating application of phylogenetic statistics, it would be natural to build on this work by a consideration of how simple projections act on tree space; which simple projections map tree space to tree space? Does it suffice to consider such maps when finding the optimal map between tree datasets? Can a simple projection which maps tree space to tree space be described by operations on leaves? These questions, which we discussed in more detail in \Cref{subsec:open_problems} present the opportunity to improve the speed of calculations through a restricted search of simple projections, as well as gaining more interpretability from our calculations on the leaf level.
	
	While our computational method in \Cref{algo:minimisation} is a natural heuristic approach and easy to implement, further research into more sophisticated optimisation tools for calculating tropical projection Wasserstein distances would significantly improve the applicability of our distance. Currently, we lack confidence bounds on our computations in \Cref{sec:computation}, and we can see immediately that our computations are not necessarily accurate; in \Cref{fig:sim_data_heatmap}, we see datasets whose distance to themselves is not zero. This may be a result of simulated annealing never reaching the optimal support as it is more common for datasets of higher dimension, but without confidence bounds our distances lack interpretability. 
	
	Finally, while tropical projection Wasserstein distances may be a useful tool for certain statistical methods such as clustering, utilising these distances for more refined model selection methods would require an understanding of their behaviour for standard tree models such as branching and coalescent processes. The computed values shown in \Cref{fig:sim_data_heatmap} do not suggest a clear pattern in the distances between datasets of the same distribution but different dimensionality, but there are many avenues for further numerical experiments. By computing distances between several samples from the same distribution, we may see some indication as to which distributions are well-clustered under the tropical projection Wasserstein distance and hence can be detected for model selection using this distance. 
	
	\paragraph{Data Availability}
	
	The code and data used in this article are publicly available in the GitHub repository at \url{https://github.com/Roroast/WDDD}.
	
	\section*{Acknowledgments}
	
	The authors wish to thank Yue Ren and Felipe Rinc\'{o}n for helpful discussions.  
	
	Y.C.~is funded by a President's PhD Scholarship at Imperial College London.  D.T.'s PhD scholarship is funded by the IGSSE/TUM-GS via a Technical University of Munich–-Imperial College London Joint Academy of Doctoral Studies (JADS) award (2021 cohort, PIs Drton/Monod), from which R.T.~also receives partial support.
	
	\begin{appendices}
		
		\crefalias{section}{appendix}
		
		\section{Proof}\label{appsec:proofs}
		
		\begin{proof}[Proof of \Cref{lemma:optimser}]
			Consider a simple projection matrix $M$, and let $M'$ be a tropical matrix map with exactly one real entry per column such that $M_{ij} \in \R \Rightarrow M'_{ij} = 0$. We begin by showing that on any ball $B_0(r)$ in $\TPT{n}$, there is some $\mathbf{t} \in \TPT{n}$ such that $M(\mathbf{x}) = M'(\mathbf{x+t})$.
			
			For each $j \in \cup_i J_i$, let $i_j$ be the unique $i$ such that $j \in J_i$. Then we can define a vector $\mathbf{t} \in \overline\R^{n}$ by
			\[
			\mathbf{t}_j = \begin{cases}
				-2r + \min\{ M_{ij}: M_{ij} \neq -\infty \}      &\text{if }j \notin  \cup_i J_i \\
				M_{i_jj}        &\text{otherwise.}
			\end{cases}
			\]
			We now show that for $\mathbf{x} \in B_0(r)$, we have $M\mathbf{x} = M'(\mathbf{x+t})$. Let $J'_i$ be the support sets for $M'$. By definition, the sets $J_i'$ are disjoint, partition $[n]$ exactly, and $J_i \subseteq J_i'$.
			\begin{align*}
				M'(\mathbf{x+t}) &= \max_{j \in J'_i} \{ M'_{ij} + x_j+t_j \} \\
				&= \max \{ \max_{j \in J_i} \{ x_j+M_{ij} \}, \max_{j \in J'_i \setminus J_i} x_j + -2r + \min\{ M_{ij}: M_{ij} \neq -\infty \} \} \\
				&= \max_{j \in J_i} \{ x_j+M_{ij} \} \\
				&= M(\mathbf{x})
			\end{align*}
			where the penultimate equality comes from $\|x\|_{tr} \leq r$.\\
			Now, denoting the empirical measure of $X$ as $\nu_X$, our set of projected measures becomes:
			\begin{equation*}
				\Phi^-(\nu_X,m) = \left\{ \frac{1}{R}\sum_{r \leq R} \delta_{M'(x^{(r)}+t)} : \mathbf{t} \in \TPT{n]}, M' \in \mathcal{M_{\tr}'} \right\}.
			\end{equation*}
			Finally, we express the Wasserstein distance $W_{\mathrm{tr},p}^{m,n}$ by
			\begin{align*}
				W_{\mathrm{tr},p}^{m,n}(\mu,\nu) &= \inf_{\mathbf{t}, M'} \inf_{\lambda} \left(\sum_{r,s} \lambda_{r,s} d_{tr}( M'( \mathbf{x^{(r)}} + \mathbf{t}), \mathbf{y})^p \right)^{1/p}\\
				&= \inf_{\mathbf{t}, M'} \inf_{\lambda} \left(\sum_{r,s} \lambda_{r,s} \| M'( \mathbf{x^{(r)}} + \mathbf{t}) - \mathbf{y}\|_{tr}^p \right)^{1/p}\\
				&= \inf_{\mathbf{t}, M'} \inf_{\lambda} f(M',\lambda, \vec t)\\
			\end{align*}
			where the constraints on $\mathbf{t},M'$ and $\lambda$ are as described.
		\end{proof}
		
		\section{Gradient Methods}\label{appsec:gradient_methods}
		
		These experiments are aimed at the computation of the 2-Wasserstein projected distance $W_{\mathrm{tr},2}^{m,n}$, and are restricted to samples $X,Y$ of equal sizes $R=S$ to reduce the computational complexity of optimal transport calculations.
		
		\begin{sidewaystable}[ht!]
			\centering
			\begin{tabular}{|c|c|c|r|r|r|r|r|r|r|r|r|r|r|}
				\hline
				&   &   & \multicolumn{11}{|c|}{Learning Rate} \\
				\hline
				$n$ & $m$ & Grad &  $e^{-6}$  &  $e^{-5}$  &  $e^{-4}$  &  $e^{-3}$  &  $e^{-2}$  &  $e^{-1}$  &  $e^{0}$  &  $e^{1}$  &  $e^{2}$  &  $e^{3}$ & $e^{4}$ \\
				\hline
				6  & 6  & CA &      2.516 &      1.927 &      1.093 &      1.086 &      1.086 &      1.088 &      1.100 &      1.173 &      1.389 &      1.895 &      3.017 \\
				&    & CD &      2.738 &      2.506 &      1.998 &      1.258 &      1.086 &      1.085 &      1.085 &      1.086 &      1.087 &      1.129 &      1.765 \\
				&    & TA &      2.699 &      2.388 &      1.638 &      1.084 &      1.086 &      1.087 &      1.090 &      1.111 &      1.173 &      1.364 &      2.521 \\
				&    & TD &      2.722 &      2.446 &      1.809 &      1.142 &      1.086 &      1.085 &      1.086 &      1.086 &      1.090 &      1.182 &      1.675 \\
				\hline
				10 & 6  & CA &      2.319 &      1.745 &      0.968 &      0.958 &      0.952 &      0.956 &      0.961 &      1.075 &      1.317 &      1.840 &      2.767 \\
				&    & CD &      2.545 &      2.320 &      1.843 &      1.228 &      0.982 &      0.971 &      0.962 &      0.958 &      0.951 &      1.026 &      1.845 \\
				&    & TA &      2.501 &      2.189 &      1.484 &      0.956 &      0.952 &      0.946 &      0.941 &      0.972 &      1.093 &      1.476 &      2.747 \\
				&    & TD &      2.530 &      2.268 &      1.655 &      1.019 &      0.963 &      0.960 &      0.957 &      0.947 &      0.939 &      1.062 &      1.726 \\
				\hline
				& 10 & CA &      2.778 &      2.199 &      1.210 &      1.195 &      1.195 &      1.201 &      1.221 &      1.313 &      1.496 &      2.003 &      3.062 \\
				&    & CD &      3.020 &      2.819 &      2.366 &      1.665 &      1.209 &      1.197 &      1.195 &      1.196 &      1.198 &      1.226 &      1.916 \\
				&    & TA &      2.963 &      2.651 &      1.929 &      1.198 &      1.195 &      1.196 &      1.204 &      1.241 &      1.320 &      1.810 &      3.415 \\
				&    & TD &      3.000 &      2.752 &      2.148 &      1.334 &      1.199 &      1.197 &      1.194 &      1.196 &      1.201 &      1.281 &      1.595 \\
				\hline
				15 & 6  & CA &      2.033 &      1.460 &      0.895 &      0.879 &      0.868 &      0.866 &      0.888 &      1.023 &      1.288 &      1.710 &      3.193 \\
				&    & CD &      2.276 &      2.081 &      1.675 &      1.081 &      0.915 &      0.903 &      0.894 &      0.884 &      0.875 &      0.973 &      1.890 \\
				&    & TA &      2.217 &      1.904 &      1.214 &      0.872 &      0.859 &      0.855 &      0.857 &      0.908 &      1.051 &      1.439 &      2.400 \\
				&    & TD &      2.255 &      2.011 &      1.459 &      0.942 &      0.891 &      0.879 &      0.869 &      0.855 &      0.855 &      1.004 &      1.794 \\
				\hline
				& 10 & CA &      2.928 &      2.303 &      1.096 &      1.080 &      1.076 &      1.074 &      1.107 &      1.238 &      1.433 &      2.014 &      3.324 \\
				&    & CD &      3.157 &      2.916 &      2.395 &      1.663 &      1.111 &      1.093 &      1.090 &      1.086 &      1.081 &      1.115 &      1.780 \\
				&    & TA &      3.118 &      2.794 &      1.982 &      1.074 &      1.069 &      1.066 &      1.073 &      1.124 &      1.282 &      1.888 &      3.435 \\
				&    & TD &      3.135 &      2.847 &      2.132 &      1.211 &      1.080 &      1.075 &      1.073 &      1.066 &      1.070 &      1.197 &      1.791 \\
				\hline
				& 15 & CA &      3.291 &      2.649 &      1.255 &      1.244 &      1.244 &      1.250 &      1.274 &      1.385 &      1.640 &      2.130 &      3.789 \\
				&    & CD &      3.529 &      3.301 &      2.800 &      2.017 &      1.300 &      1.249 &      1.244 &      1.242 &      1.244 &      1.270 &      1.620 \\
				&    & TA &      3.484 &      3.155 &      2.306 &      1.245 &      1.243 &      1.245 &      1.256 &      1.306 &      1.400 &      2.182 &      3.993 \\
				&    & TD &      3.505 &      3.216 &      2.487 &      1.435 &      1.252 &      1.246 &      1.244 &      1.244 &      1.251 &      1.318 &      1.732 \\
				\hline
				\hline
				\multicolumn{2}{|c|}{Average} & CA   &      2.644 &      2.047 &      1.086 &      1.074 &      \textbf{1.070} &      1.073 &      1.092 &      1.201 &      1.427 &      1.932 &      3.192 \\
				\multicolumn{2}{|c|}{} 		& CD   &      2.878 &      2.657 &      2.179 &      1.485 &      1.101 &      1.083 &      1.078 &      1.075 &      \textbf{1.073} &      1.123 &      1.803 \\
				\multicolumn{2}{|c|}{} 		& TA   &      2.830 &      2.514 &      1.759 &      1.072 &      1.067 &      \textbf{1.066} &      1.070 &      1.110 &      1.220 &      1.693 &      3.085 \\
				\multicolumn{2}{|c|}{} 		& TD   &      2.858 &      2.590 &      1.948 &      1.180 &      1.078 &      1.074 &      1.071 &      \textbf{1.066} &      1.068 &      1.174 &      1.719 \\
				\hline
			\end{tabular}
			\caption{The mean loss over 10 initialisations and 9 sample distributions achieved by classical Adam, classical descent, TrAdamax, and tropical descent for data of dimensionality $m \leq n \in \{ 6, 10, 15 \}$. Learning rate is varied on a linear log scale from $e^{-6}, e^{4}$.}
			\label{tab:lr_tuning}
		\end{sidewaystable}
		
		We test the gradient methods CD, TD, CA, TA in the following way; for each dimension pair $m \leq n \in \{6, 10,15 \} $, we have 9 dataset pairs $(X,Y)$ where $X,Y$ have been taken from a branching, coalescent or Gaussian distribution. We then run \Cref{algo:minimisation} with 10 random initialisations for $\texttt{max\_iters} = 40$. Each iteration of \Cref{algo:minimisation} involves 50 steps of the specified gradient method -- we do not perform any $\lambda, M$ optimisation to minimise the random noise present. We do this for each dataset and gradient method for learning rates on a log scale between $e^{-6}, e^{4}$. \Cref{tab:lr_tuning} shows the mean loss achieved over 10 randomised initialisations and all 9 dataset pairs for each dimension pair and gradient settings.
		
		We see that when $n=m$, classical and tropical methods behave similarly; their optimal mean loss is the same, though this minimum is achieved for different learning rates. However, when $m < n$ the tropical methods consistently achieve a smaller mean loss. \Cref{tab:lr_tuning} also shows the mean loss averaged over all dimension pairs; the globally optimum gradient settings are TA, $\alpha=0.3679$.
		
		\section{Simulated Annealing}\label{appsec:simulated_annealing}
		
		\begin{sidewaystable}[h!]
			\centering
			\begin{tabular}{|l|l|l|r|r|r|r|r|r|r|r|r|r|}
				\hline
				&   &   & \multicolumn{10}{|c|}{Scale, $\sigma$} \\
				\hline
				$n$ & $m$ & Graph &  $e^{-7}$  &  $e^{-6}$  &  $e^{-5}$  &  $e^{-4}$  &  $e^{-3}$  &  $e^{-2}$  &  $e^{-1}$  &  $e^{0}$  &  $e^{1}$  &  $e^{2}$ \\
				\hline
				6  & 6  & K &     1.087 &     1.083 &     1.075 &     1.079 &     1.087 &     1.080 &     1.057 &     1.044 &     1.509 &     1.593 \\
				&    & G &     1.076 &     1.076 &     1.076 &     1.076 &     1.075 &     1.062 &     1.042 &     1.122 &     1.431 &     1.640 \\
				\hline
				10 & 6  & K &     0.946 &     0.946 &     0.946 &     0.946 &     0.946 &     0.943 &     0.918 &     0.986 &     1.487 &     1.561 \\
				&    & G &     0.925 &     0.926 &     0.925 &     0.925 &     0.910 &     0.890 &     0.898 &     1.034 &     1.349 &     1.505 \\
				\hline
				& 10 & K &     1.196 &     1.196 &     1.196 &     1.196 &     1.196 &     1.194 &     1.195 &     1.320 &     1.666 &     1.700 \\
				&    & G &     1.188 &     1.186 &     1.185 &     1.187 &     1.183 &     1.160 &     1.199 &     1.602 &     1.669 &     1.716 \\
				\hline
				15 & 6  & K &     0.855 &     0.855 &     0.855 &     0.855 &     0.855 &     0.852 &     0.848 &     0.936 &     1.395 &     1.438 \\
				&    & G &     0.825 &     0.826 &     0.824 &     0.818 &     0.818 &     0.816 &     0.839 &     1.070 &     1.287 &     1.369 \\
				\hline
				& 10 & K &     1.066 &     1.066 &     1.066 &     1.066 &     1.066 &     1.064 &     1.059 &     1.162 &     1.616 &     1.643 \\
				&    & G &     1.064 &     1.064 &     1.066 &     1.064 &     1.058 &     1.047 &     1.136 &     1.486 &     1.542 &     1.585 \\
				\hline
				& 15 & K &     1.245 &     1.245 &     1.245 &     1.245 &     1.245 &     1.242 &     1.235 &     1.442 &     1.686 &     1.701 \\
				&    & G &     1.242 &     1.241 &     1.241 &     1.239 &     1.221 &     1.221 &     1.492 &     1.660 &     1.690 &     1.706 \\
				\hline
				\hline
				\multicolumn{2}{|c|}{Averages}& K   &     1.066 &     1.065 &     1.064 &     1.064 &     1.066 &     1.062 &     1.052 &     1.148 &     1.560 &     1.606 \\
				\multicolumn{2}{|c|}{}& G &     1.053 &     1.053 &     1.053 &     1.052 &     1.044 &     1.033 &     1.101 &     1.329 &     1.495 &     1.587 \\
				\hline
			\end{tabular}
			\caption{The mean loss over 10 initialisations and 9 sample distributions achieved by simulated annealing on the incomplete and complete graph structures for data of dimensionality $m \leq n \in \{ 6, 10, 15 \}$. The scale factor is varied on a linear log scale from $e^{-7}, e^{2}$.}
			\label{tab:scale_tuning}
		\end{sidewaystable}
		
		As for the training of gradient methods for $\vec t$ optimisation, we run experiments for the 54 pairs of datasets $X,Y$. \Cref{algo:minimisation} is run with 10 random initialisations for $\texttt{max\_iters} = 40$ and for each $k$,  $\vec t$ optimisation is run for 40 steps using TrAdamax and $\alpha = 0.3679$. For each $k$ we perform 50 steps of simulated annealing, over which the scale factor decreases linearly from $T = \sigma/\sqrt{k}$ to $\sigma/50\sqrt{k}$. We perform this computation for the incomplete graph structure described above as well as the complete graph of supports, and for scale factors on a log scale between $e^{-6}, e^2$. The mean loss are then shown in \Cref{tab:scale_tuning}.
		
		We see that for data of any dimensionality, simulated annealing on the incomplete graph $G$ outperforms simulated annealing on the complete graph $K$. We achieve the best mean loss when using the incomplete graph and the scale factor $\sigma = 0.135$.
		
		\section{Runtimes}\label{appsec:time}
		
		Experiments were implemented on a AMD EPYC 7742 node using a single core, 8 GB. The average runtime over all simulated dataset computations was 2.90s. The average runtime over all real dataset computations was 3.66s.

	\end{appendices}
	
	\clearpage
	
	
	\bibliographystyle{authordate3}
	
	\bibliography{Sources}
	
\end{document}